\documentclass[reqno,11pt]{article}
\usepackage{mathrsfs}
\usepackage{amssymb}
\usepackage{amsthm}
\usepackage{amsmath}
\usepackage{amsfonts}
\usepackage{mathtools}
\usepackage{enumerate}
\usepackage{bm}

\usepackage{graphicx}

\usepackage{bbm}
\usepackage{mathrsfs}
\usepackage{amssymb}
\usepackage[thicklines]{cancel}

\UseRawInputEncoding

\usepackage[usenames,dvipsnames]{xcolor}
\usepackage{soul}

\usepackage{txfonts}

\usepackage{psfrag}
\usepackage[numbers,sort&compress]{natbib}
\usepackage{setspace,color,wrapfig}
\usepackage[colorlinks,linkcolor=blue,citecolor=cyan]{hyperref}

\numberwithin{equation}{section}
\newtheorem{theorem}{Theorem}[section]
\newtheorem{lemma}[theorem]{Lemma}

\newtheorem{proposition}[theorem]{Proposition}
\newtheorem{problem}[theorem]{Problem}
\theoremstyle{definition}
\newtheorem{definition}[theorem]{Definition}

\newtheorem{remark}[theorem]{Remark}
\theoremstyle{remark}

\begin{document}
\title{Subsonic Euler-Poisson flows with nonzero vorticity in convergent  nozzles }
\author{Yuanyuan Xing\thanks{School of Mathematics and Statistics,
                Northeastern University at Qinhuangdao, Qinhuangdao,  Hebei Province, 066004, China. Email: xingyuanyuan@neuq.edu.cn }\and
Zihao Zhang\thanks{School of Mathematics, Jilin University, Changchun, Jilin Province,  130012, China. Email: zhangzihao@jlu.edu.cn}}
\date{}
\newtheorem{pro}{Problem}[section]
 \newtheorem{thm}{Theorem}[section]
 \newtheorem{cor}[thm]{Corollary}
 \newtheorem{lem}[theorem]{Lemma}
 \newtheorem{prop}[thm]{Proposition}
 \newtheorem{defn}{Definition}[section]
  \theoremstyle{remark}
 \newtheorem{rem}{\bf{Remark}}[section]
 \numberwithin{equation}{section}

\def\be{\begin{equation}}
\def\ee{\end{equation}}
\def\beq{\begin{equation*}}
\def\eeq{\end{equation*}}
\def\bc{\begin{cases}}
\def\ec{\end{cases}}
\def\bs{\begin{split}}
\def\es{\end{split}}
\def\oT{\overline{T}}
\def\hT{\hat{T}}
\def\d{{\rm div}}
\def\R{\mathbb{R}^2}
\def\c{\cdot}
\def\n{\nabla}
\def\i{\infty}
\def\D{\mathcal{D}(u)}
\def\p{\partial}

\renewcommand\figurename{\scriptsize Fig}
\pagestyle{myheadings} \markboth{ Subsonic Euler-Poisson flows  with
nonzero  vorticity }{ Subsonic Euler-Poisson flows with
nonzero  vorticity }\maketitle
\begin{abstract}
  This paper concerns the well-posedness of subsonic Euler-Poisson flows in a convergent nozzle. Due to the geometry of the nozzle, we first introduce a coordinate  transformation to prove
 the existence of radially symmetric subsonic solutions to the steady Euler-Poisson system. We then
 investigate the structural stability of these background subsonic flows under perturbations of suitable boundary conditions, and establish the existence and uniqueness of smooth subsonic Euler-Poisson flows with nonzero vorticity.   The solution shares the same regularity for the velocity, the pressure, the entropy and the electric potential.  The deformation-curl-Poisson decomposition is utilized to
reformulate the steady Euler-Poisson system as a deformation-curl-Poisson system
together with several transport equations. The key point lies on the analysis of the well-posedness of the boundary value problem for the associated linearized elliptic system, which is established by using a special structure of the system to derive a priori estimates.
\end{abstract}
\begin{center}
\begin{minipage}{5.5in}
\end{minipage}
\end{center}
\section{Introduction and main results}\noindent
\par In this paper, we investigate  smooth subsonic Euler-Poisson   flows with nonzero vorticity in  a convergent nozzle, which is governed by the two-dimensional steady compressible Euler-Poisson system:
\begin{equation}\label{ep1}
\begin{cases}
\begin{aligned}
&(\rho u)_x+(\rho v)_y=0,\\
&(\rho u^2)_x+(\rho uv)_y+P_x=\rho \Phi_x,\\
&(\rho uv)_x+(\rho v^2)_y+P_y=\rho \Phi_y,\\
&(\rho u\mathcal{E}+Pu)_x+(\rho v\mathcal{E}+Pv)_y=\rho \mathbf{u}\cdot\nabla\Phi,\\
&\Delta\Phi=\rho-\tilde b,
\end{aligned}
\end{cases}
\end{equation}
where  $\rho$, $\mathbf{u}=(u,v)$, $P$, $\mathcal{E}$ and $\Phi$ represent the density, velocity field, pressure, total energy and electric potential, respectively. $\tilde b>0$ denotes the density of fixed, positively charged background ions.
 For the ideal polytropic gas, the equation of state and
 the total energy  are
of the form
\begin{equation*}
  P= e^{S}\rho^{\gamma}
\quad{\rm {and}}\quad \mathcal{E}=\frac1 2|{\bf u}|^2+\frac{ P}{(\gamma-1)\rho},
\end{equation*}
respectively, where $S$ is  the entropy and $ \gamma> 1 $ is the adiabatic exponent. Denote the Mach number by
$$M=\frac{\vert \mathbf{u}\vert}{c},$$ where $c=\sqrt{\frac{\gamma P}{\rho} }$ is the local sound speed. The solution to \eqref{ep1} is called supersonic if $M>1$, subsonic if $M<1$, and sonic if $M=1$.
\par The Euler-Poisson system stands as one of the fundamental hydrodynamical models describing the dynamic behaviour of many important physical flows, including the propagation of electrons in submicron semiconductor devices \cite{Markowich}, the biological transport of ions for channel proteins \cite{Chen}, the motion of self-gravitating gaseous stars \cite{Chandrasekhar} and so on. In the study of time-independent solutions to the Euler-Poisson system, the well-posedness theory of nozzle flows has always been one of the key issues in the field of partial differential equations.

The main difficulty in the mathematical research of multi-dimensional Euler-Poisson system lies in the fact that the system changes type as the flow speed varies from subsonic to supersonic. For a subsonic state, the system can be decomposed into a nonlinear elliptic system with homogeneous transport equations. The first result concerning subsonic solutions with large variations for the Euler-Poisson system was given in \cite{Bae3}, where the unique existence and stability of isentropic potential flows in a multidimensional flat nozzle were established by exploiting a key structural property of the elliptic system to derive an energy estimate. Subsonic flows with nonzero vorticity in a two-dimensional flat nozzle were investigated in \cite{Bae2} by combining the Helmholtz decomposition with the stream function approach. A new Helmholtz decomposition for the velocity field was introduced in \cite{Bae5} to establish the structural stability of three-dimensional axially symmetric subsonic flows with nonzero swirl in a circular cylinder. Similar results for self-gravitating isentropic and non-isentropic subsonic flows in two-dimensional flat nozzles and annuluses were obtained in \cite{Bae4,Cao,Duan,WZ25} by using the stream function formulation. For a supersonic state, the Euler-Poisson system can be decomposed into a nonlinear hyperbolic-elliptic coupled system with homogeneous transport equations. The structural stability of supersonic irrotational flows and flows with nonzero vorticity in a two-dimensional flat nozzle was proved in \cite{Bae1}, where a priori estimate for the associated linearized hyperbolic-elliptic system was obtained through a delicate choice of multipliers. Supersonic potential flows in three-dimensional  cylindrical nozzles and two-dimensional divergent nozzles were investigated in \cite{Bae6,Duan0,WZ25-1}. For other relevant works, one may refer to \cite{Degond1,Degond2,Guo,Markowich0}.
\par  It is shown in \cite{Liu} that transonic shock flows governed by Euler system are dynamically stable in a divergent nozzle but unstable in a convergent nozzle. Furthermore, \cite{Luo1,Luo2} reveals that if the background charge density is less than the sonic density, the flows in one-dimensional flat nozzles governed by the Euler-Poisson system exhibit behavior similar to flows in divergent nozzles for the Euler system. These results indicate that the geometry of the domain and the electric potential play a crucial role in determining the analytic behavior of nozzle flows. This naturally raises a question of the well-posedness of multi-dimensional flows under the coupled effects of the electric potential and the geometry of a convergent domain.
The main goal of this paper is to establish the unique existence and structural stability of subsonic flows with nonzero vorticity to the  steady Euler-Poisson system through a two-dimensional convergent nozzle  by characterizing a class of physical acceptable boundary conditions. Our results will demonstrate that the electric field force in compressible subsonic flows can counteract the geometric effects of the convergent nozzle, thereby stabilizing certain physical features of the flow. On the other hand,
the study
  for subsonic flows to Euler-Poisson system is a crucial step to analyze supersonic flows with nonzero vorticity under the coupled effects of the electric potential and the geometry of a convergent domain.
  \subsection{Radially symmetric subsonic  Euler-Poisson flows}\noindent
\par Let $\theta_0\in (0, \frac{\pi}{2})$, $r_1$, $r_2(>r_1)$ be  fixed positive constants. Define   a two-dimensional convergent  nozzle by
\begin{equation*}
\tilde\Omega=\Big\{(x,y):0<r_1<\sqrt{x^2+y^2}<r_2<+\infty, -\theta_0<\arctan\frac{y}{x}<\theta_0\Big\}.
\end{equation*}
 In the polar coordinates $(\tilde{r},\tilde{\theta})$ satisfying $x=\tilde{r}\cos\theta$ and $y=\tilde{r}\sin\theta$, the velocity is decomposed as ${\mathbf{u}}(x,y)=\tilde{U}\mathbf{e}_{\tilde{r}}+\tilde{V}\mathbf{e}_{\tilde{\theta}}$ with
\begin{equation}\label{coef2}
\begin{aligned}
\mathbf{e}_{\tilde{r}}=
\left(
  \begin{array}{c}
    \cos\tilde{\theta} \\
    \sin\tilde{\theta} \\
  \end{array}
\right)
,\quad \mathbf{e}_{\tilde{\theta}}=
\left(
  \begin{array}{c}
    -\sin\tilde{\theta} \\
    \cos\tilde{\theta} \\
  \end{array}
\right).
\end{aligned}
\end{equation}
Then \eqref{ep1} is transformed into
\begin{equation}\label{EP}
\begin{cases}
\begin{aligned}
&(\tilde{\rho} \tilde{U})_{\tilde{r}}+\frac{1}{\tilde{r}}(\tilde{\rho} \tilde{V})_{\tilde{\theta}}+\frac{1}{\tilde{r}}\tilde{\rho} \tilde{U}=0,\\
&\tilde{\rho} \tilde{U}\tilde{U}_{\tilde{r}}+\frac{1}{\tilde{r}}\tilde{\rho} \tilde{V}\tilde{U}_{\tilde{\theta}}-\frac{1}{\tilde{r}}\tilde{\rho} \tilde{V}^2+\tilde{P}_{\tilde{r}}=\tilde{\rho}\tilde{\Phi}_{\tilde{r}},\\
&\tilde{\rho} \tilde{U}\tilde{V}_{\tilde{r}}+\frac{1}{\tilde{r}}\tilde{\rho} \tilde{U}\tilde{V}+\frac{1}{\tilde{r}}\tilde{\rho} \tilde{V}\tilde{V}_{\tilde{\theta}}+\frac{1}{\tilde{r}}\tilde{P}_{\tilde{\theta}}
=\frac{1}{\tilde{r}}\tilde{\rho}\tilde{\Phi}_{\tilde{\theta}},\\
&(\tilde{\rho} \tilde{U}\tilde{\mathscr{B}})_{\tilde{r}}+\frac{1}{\tilde{r}}(\tilde{\rho} \tilde{V}\tilde{\mathscr{B}})_{\tilde{\theta}}+\frac{1}{\tilde{r}}\tilde{\rho} \tilde{U}\tilde{\mathscr{B}}
=\tilde{\rho}(\tilde{U}\tilde{\Phi}_{\tilde{r}}+\frac{1}{\tilde{r}}\tilde{V}\tilde{\Phi}_{\tilde{\theta}}),\\
&\tilde{\Phi}_{\tilde{r}\tilde{r}}+\frac{1}{\tilde{r}^2}\tilde{\Phi}_{\tilde{\theta}\tilde{\theta}}
+\frac{1}{\tilde{r}}\tilde{\Phi}_{\tilde{r}}=\tilde{\rho}-\tilde b,
\end{aligned}
\end{cases}
\end{equation}
where
\begin{equation}
\tilde{\mathscr{B}}=\frac{\tilde{U}^2+\tilde{V}^2}{2}+\frac{\gamma \tilde{P}}{(\gamma-1)\tilde{\rho}}.
\end{equation}
 We define the pseudo-Bernoulli function $\tilde{\mathscr{K}}$ by
\begin{equation*}
\tilde{\mathscr{K}}=\tilde{\mathscr{B}}-\tilde{\Phi}.
\end{equation*}
Then it is derived from \eqref{EP} that
\begin{equation}\label{EP-k}
(\tilde{\rho} \tilde{U}\tilde{\mathscr{K}})_{\tilde{r}}+\frac{1}{\tilde{r}}(\tilde{\rho} \tilde{V}\tilde{\mathscr{K}})_{\tilde{\theta}}+\frac{1}{\tilde{r}}\tilde{\rho} \tilde{U}\tilde{\mathscr{K}}
=0.
\end{equation}

 \par We first construct a class of radially symmetric subsonic flows to the steady Euler-Poisson system in a convergent nozzle. Fix $\tilde b$ to be a constant $b_0>0$,  the background flow is described by smooth functions of the form
\begin{equation*}
{\bf u}(x,y)=\tilde U(\tilde{r})\mathbf e_r,
 \ \ \tilde\rho(x,y)=\tilde \rho(\tilde{r}),\ \  \tilde P(x,y)=\tilde P(\tilde{r}),\ \  \tilde\Phi(x,y)=\tilde \Phi(\tilde{r}), \ \ \tilde r\in[r_1, r_2],\\
\end{equation*}
which solves
\begin{equation}\label{sreq}
\begin{cases}
\frac{\mathrm{d}}{\mathrm{d}\tilde r}(\tilde{r}\tilde{\rho} \tilde{U})=0, &  \tilde r\in[r_1, r_2],\\
\frac{\mathrm{d}}{\mathrm{d}\tilde r}(\tilde{r}\tilde{\rho} \tilde{U}^2)+\tilde{r}\frac{\mathrm{d}\tilde{P}}{\mathrm{d}\tilde r}=\tilde{r}\tilde{\rho} \tilde{E},&  \tilde r\in[r_1, r_2],\\
\frac{\mathrm{d}}{\mathrm{d}\tilde r}(\tilde{r}\tilde{\rho} \tilde{U}\tilde{\mathscr{B}})=\tilde{r}\tilde{\rho} \tilde{U}\tilde{E},&  \tilde r\in[r_1, r_2],\\
\frac{\mathrm{d}}{\mathrm{d}\tilde r}(\tilde{r}\tilde{E})=\tilde{r}(\tilde{\rho}-b_0), &  \tilde r\in[r_1, r_2],\\
\end{cases}
\end{equation}
with the boundary conditions
\begin{equation}\label{srco}
\tilde{\rho}(r_2)=\rho_0,\quad \tilde{U}(r_2)=-U_0, \quad \tilde{P}(r_2)=P_0,\quad \tilde{E}(r_2)=E_0.
\end{equation}
 Here
$$  \tilde E( \tilde r)=\frac{\mathrm{d}\tilde\Phi}{\mathrm{d}\tilde r} \ \ {\rm{and}} \ \ \tilde{\mathscr{B}}( \tilde r)=\frac{\tilde U^2}{2}+\frac{\gamma \tilde P}{(\gamma-1)\tilde \rho}, \ \  \tilde r\in[r_1, r_2].$$
 As we seek solutions flowing in the direction of $-\mathrm{e}_{\tilde{r}}$, introduce new variable $(r,\hat{r})=(r_2-\tilde{r},\tilde{r})$ and set
\begin{equation}
(\bar\rho,\bar U, \bar P, \bar E)(r):=(\tilde{\rho}, -\tilde{U}, \tilde{P}, \tilde{E})(\tilde{r}),\ \ r\in[0,r_2-r_1].
\end{equation}
Denote $\frac{\mathrm{d}}{\mathrm{d}r}$ by $'$ and $ R=r_2-r_1 $. \eqref{sreq} and \eqref{srco} are equivalent to the following equations for $(\rho, U, P, E)$:
\begin{equation}\label{-sreq}
\begin{cases}
\begin{aligned}
&(\hat{r}\bar\rho \bar U)'=0,\ \ &r\in[0,R],\\
&(\hat{r}\bar\rho  \bar U^2)'+\hat{r}\bar P'=-\hat{r}\bar\rho\bar E,\ \ &r\in[0,R],\\
&(\hat{r}\bar\rho\bar U\bar{\mathscr{B}})'=-\hat{r}\bar\rho U\bar E,\ \ &r\in[0,R],\\
&(\hat{r}\bar E)'=-\hat{r}(\bar\rho-b_0),\ \ &r\in[0,R],
\end{aligned}
\end{cases}
\end{equation}
with the initial value
\begin{equation}\label{-srco}
(\bar \rho,\bar  U,\bar  P,\bar  E)(0)=(\rho_0,U_0,P_0,E_0),
\end{equation}
where
\begin{equation}
\bar{\mathscr{B}}(r)=\frac{\bar U^2}{2}+\frac{\gamma e^{\bar S} \bar\rho^{\gamma-1}}{\gamma-1},\ \ r\in[0,R].
\end{equation}
It follows from \eqref{-sreq} that
\begin{equation}\label{J-0}
\begin{cases}
\begin{aligned}
&\hat{r}(\bar\rho\bar U)(r)=J_0,\quad J_0=r_2\rho_0U_0,\ \ &r\in[0,R],\\
&\bar S(r)=S_0,\quad S_0=\ln\frac{P_0}{\rho_0^{\gamma}},\ \ &r\in[0,R],\\
&\bar{\mathscr{B}}'=-\bar E,\ \ &r\in[0,R].
\end{aligned}
\end{cases}
\end{equation}
Furthermore, it holds that
\begin{equation}\label{J-00}
\begin{split}
\bar\rho'=\frac{\bar\rho \bar E}{\bar c^2(\bar M^2-1)}+\frac{\bar\rho \bar M^2}{\hat{r}(\bar M^2-1)},\quad
\bar U'=\frac{\bar U\bar E}{\bar c^2(1-\bar M^2)}+\frac{\bar U}{\hat{r}(1-\bar M^2)}, \ \ \bar M^2=\frac{\bar U^2}{\bar c^2}.
\end{split}
\end{equation}
 The definition of $\bar M^2$ gives that
\begin{equation}
\bar\rho=\left(\frac{J_0^2}{\gamma\mathrm{e}^{S_0}\hat{r}^2\bar M^2}\right)^{\frac{1}{\gamma+1}}
:=\mu_0\left(\frac{1}{\hat{r}^2\bar M^2}\right)^{\frac{1}{\gamma+1}} \ \ {\rm{and}} \ \
\bar c^2=\gamma\mathrm{e}^{S_0}\bar \rho^{\gamma-1}
=\gamma\mathrm{e}^{S_0}\mu_0^{\gamma-1}
\left(\frac{1}{\hat{r}^2\bar M^2}\right)^{\frac{\gamma-1}{\gamma+1}}.
\end{equation}
Then direct calculations deduce that
\begin{equation}\label{M2}
\begin{split}
(\bar M^2)'&=\frac{\bar M^2}{1-\bar M^2}\left(\frac{(\gamma+1)\bar E}{\bar c^2}+\frac{2+(\gamma-1)\bar M^2}{\hat{r}}\right)\\
&=\frac{\bar M^2}{1-\bar M^2}
\left(\frac{(\gamma+1)\bar E}{\gamma\mathrm{e}^{S_0}\mu_0^{\gamma-1}}(\hat{r}^2\bar M^2)
^{\frac{\gamma-1}{\gamma+1}}+\frac{2+(\gamma-1)\bar M^2}{\hat{r}}\right)\\
&:=h_1(r,\bar M^2,\bar E),
\end{split}
\end{equation}
and
\begin{equation}\label{rE}
\begin{split}
(\hat{r}\bar E)'=-\hat{r}\left(\mu_0\left(\frac{1}{\hat{r}^2\bar M^2}\right)^{\frac{1}{\gamma+1}}-b_0\right):=h_2(r,\bar M^2,\bar E).
\end{split}
\end{equation}
This, together with \eqref{J-0} and \eqref{J-00}, shows that the initial value problem \eqref{-sreq}-\eqref{-srco} is equivalent to
\begin{equation}\label{sreq1}
\begin{cases}
\begin{aligned}
&\hat{r}(\bar\rho \bar U)(r)=J_0, \quad \bar S(r)=S_0,\ \ &r\in[0,R],\\
&(\bar M^2)'=h_1(r,\bar M^2,\bar E),\ \ &r\in[0,R],\\
&(\hat{r}\bar E)'=h_2(r,\bar M^2,\bar E),\ \ &r\in[0,R],
\end{aligned}
\end{cases}
\end{equation}
with
the boundary conditions:
\begin{equation}\label{-srco-c}
(\bar M^2,\bar E)(0)=(M_0^2,E_0)=\bigg(\frac{\rho_0U_0^2}{\gamma P_0},E_0\bigg).
\end{equation}
  Inspired by \cite{Bae7},
we establish the existence and uniqueness of smooth subsonic solution to \eqref{-sreq} and \eqref{-srco} in the following lemma.
\begin{lemma}\label{Tm}
Fix $\gamma>1 $ and $ b_0>0$. For given positive constants $ r_2 $, $\rho_0$, $U_0$ and $P_0$ satisfying
$$\frac{\rho_0U_0^2}{\gamma P_0}<1, $$ if $r_1\in(0,r_2)$ satisfies
\begin{equation}\label{LE}
\ln\frac{r_2}{r_1}<\frac{\gamma+1}{2(\gamma-1)},
\end{equation}
then there exists a constant $\check{E}<0$ depending only on the data $(r_1, r_2, \gamma,b_0, \rho_0,U_0,P_0 )$ such that, for any $E_0<\check{E}$, the initial value problem \eqref{sreq1} and \eqref{-srco-c} has a unique smooth subsonic solution $(\bar M^2,\bar E)(r)$ for $r\in[0,R]$, provided $R=r_2-r_1$.
\end{lemma}
\begin{proof}
We rewrite \eqref{M2} as
\begin{equation}
\begin{split}
(\bar M^2)'&=\frac{(\gamma+1)\hat{r}^{\frac{\gamma-3}{\gamma+1}}(\bar M^2)^{\frac{2\gamma}{\gamma+1}}}
{\gamma\mathrm{e}^{S_0}\mu_0^{\gamma-1}(1-\bar  M^2)}
\left(\hat{r}\bar  E+\frac{2(\gamma-1)}{\gamma+1}
\left(\frac{\bar U^2}{2}+\frac{\gamma\mathrm{e}^{S_0}\bar \rho^{\gamma-1}}{\gamma-1}\right)\right)\\
&=\frac{(\gamma+1)\hat{r}^{\frac{\gamma-3}{\gamma+1}}(\bar M^2)^{\frac{2\gamma}{\gamma+1}}}
{\gamma\mathrm{e}^{S_0}\mu_0^{\gamma-1}(1-\bar  M^2)}
\left(\hat{r}\bar  E+\frac{2(\gamma-1)}{\gamma+1}\bar {\mathscr{B}}\right).
\end{split}
\end{equation}
According to the local existence and uniqueness theory of ODE system, the initial value problem \eqref{sreq1} and \eqref{-srco-c} with $M_0^2<1 $ admits a unique smooth subsonic solution on $[0,\epsilon_0]$, where $\epsilon_0>0$ is some small constant.
Note that $\bar M^2(r)$ is continuous on $[0,\epsilon_0]$, there exists a constant $M^*$ such that $0<M^*\leq \bar M^2<1$.
It follows from \eqref{rE} that
\begin{equation}\label{reqE}
-\hat{r}C_0\left(\frac{1}{\hat{r}}\right)^{\frac{2}{\gamma+1}}<(\hat{r}\bar E)'<\hat{r}b_0 \quad \hbox{for} \quad r\in[0,\epsilon_0),
\end{equation}
provided $C_0=\mu_0\left(\frac{1}{M^*}\right)^{\frac{1}{\gamma+1}}$. Integrating this inequality over $[0,r)$ for $r\in[0,\epsilon_0)$ deduces that
\begin{equation}
\frac{1}{r_2-r}\left(r_2E_0-\int_{0}^{r}(r_2-\nu)C_0
\left(\frac{1}{r_2-\nu}\right)^{\frac{2}{\gamma+1}}\mathrm{d}\nu\right)
<\bar  E(r)
<\frac{r_2E_0+r_2b_0r-\frac{b_0}{2}r^2}{r_2-r}.
\end{equation}
There exists a constant $C_1>0$ depending only on the data $(r_2,  \gamma,b_0, \rho_0,U_0,P_0)$ such that
\begin{equation}\label{bu}
\frac{r_2E_0-C_1}{r_2-r}\leq \bar  E(r)<\frac{r_2E_0+C_1}{r_2-r}.
\end{equation}
A direct combination of \eqref{J-0}, \eqref{J-00} and \eqref{rE} yields
\begin{equation}
\begin{split}
\left(\hat{r}\bar  E+\frac{2(\gamma-1)}{\gamma+1}\bar {\mathscr{B}}\right)'
&=-\hat{r}\left(\mu_0\left(\frac{1}{\hat{r}^2\bar M^2}\right)^{\frac{1}{\gamma+1}}-b_0\right)
-\frac{2(\gamma-1)}{\gamma+1}\bar E\\
&<r_2b_0-\frac{2(\gamma-1)}{\gamma+1}\frac{r_2E_0-C_1}{r_2-r}.
\end{split}
\end{equation}
Consequently,
\begin{equation}
\begin{split}
\hat{r}\bar E+\frac{2(\gamma-1)}{\gamma+1}\bar {\mathscr{B}}
\leq &r_2E_0+\frac{2(\gamma-1)}{\gamma+1}\mathscr{B}_0
+\int_{0}^{r}\left(r_2b_0-\frac{2(\gamma-1)}{\gamma+1}\cdot\frac{r_2E_0-C_1}{r_2-\nu}\right)\mathrm{d}\nu\\
=&r_2E_0\left(1-\frac{2(\gamma-1)}{\gamma+1}\ln\frac{r_2}{r_2-r}\right)+\frac{2(\gamma-1)}{\gamma+1}\mathscr{B}_0\\
&+\int_{0}^{r}\left(r_2b_0+\frac{2(\gamma-1)}{\gamma+1}\frac{C_1}{r_2-\nu}\right)\mathrm{d}\nu\\
\leq &r_2E_0\left(1-\frac{2(\gamma-1)}{\gamma+1}\ln\frac{r_2}{r_2-r}\right)+C_2 \quad \hbox{for}\quad r\in[0,\epsilon_0],
\end{split}
\end{equation}
where $\mathscr{B}_0=\frac{U_0^2}{2}+\frac{\gamma e^{S_0} \rho_0^{\gamma-1}}{\gamma-1}$ and $C_2>0$ is a constant depending only on the data $(r_2, \gamma, b_0, \rho_0,U_0,P_0)$. If $E_0>0$, then
\begin{equation}
\begin{split}
\hat{r}\bar E+\frac{2(\gamma-1)}{\gamma+1}\bar {\mathscr{B}}
\leq r_2E_0+C_2,
\end{split}
\end{equation}
provided $r_2E_0+C_2$ is positive. It does not guarantee that $(M^2)'<0$ for all $r\in[0,\epsilon_0)$.
Therefore, we assume that $E_0<0$. Then
\begin{equation}
\begin{split}
\hat{r}\bar E+\frac{2(\gamma-1)}{\gamma+1}\bar {\mathscr{B}}
\leq r_2E_0\left(1-\frac{2(\gamma-1)}{\gamma+1}\ln\frac{r_2}{r_2-\epsilon_0}\right)+C_2.
\end{split}
\end{equation}
Choose constants $0<r_1<r_2$ such that \eqref{LE} is satisfied. And set
\begin{equation}
\delta_0:=1-\frac{2(\gamma-1)}{\gamma+1}\ln\frac{r_2}{r_1}
\end{equation}
so that $0<\delta_0<1$ holds. Then
\begin{equation}
r_2E_0\left(1-\frac{2(\gamma-1)}{\gamma+1}\ln\frac{r_2}{r_2-\epsilon_0}\right)+C_2\leq r_2E_0\delta_0+C_2
\end{equation}
whenever $\epsilon_0\in(0,r_2-r_1)$. Set $\check{E}:=-\frac{C_2}{r_2\delta_0}$ depending only on $(r_2, \gamma,b_0, \rho_0,U_0,P_0)$. If $E_0<\check{E}$, then
\begin{equation}\label{mono}
(\bar M^2)'(r)<0\quad \hbox{for}\quad r\in[0,\epsilon_0).
\end{equation}
Consequently, for given $M_0^2<1$ and $E_0<\check{E}$, the initial value problem \eqref{sreq1} and   \eqref{-srco-c}   has a unique smooth subsonic solution $(\bar M^2,\bar E)(r)$ satisfying $(\bar M^2)'<0$ on $[0,\epsilon_0)$. With the aid of \eqref{reqE} and \eqref{bu}, $\bar E$ can be extended up to $r=\epsilon_0$ as
\begin{equation}
(r_2-\epsilon_0)E(\epsilon_0)=r_2E(0)+\lim_{r\to\epsilon_0}\int_{0}^{r}h_2(\nu,
\bar M^2(\nu),\bar E(\nu))\mathrm{d}\nu.
\end{equation}
There exists a constant $C_3>0$ depending only on $( r_1, r_2, \gamma,b_0, \rho_0,U_0,P_0)$ such that
\begin{equation}
\sup_{r\in[0,\epsilon_0]}\vert \bar E(r)\vert\leq C_3.
\end{equation}
Then \eqref{M2} implies that
\begin{equation}
\frac{M_0^2}{M_0^2-1}
(\gamma+1)C_4\frac{(r_2^2)^{\frac{\gamma-1}{\gamma+1}}}{r_1}\leq\frac{(\bar M^2)'}{(\bar M^2)^{\frac{\gamma-1}{\gamma+1}}}<0,
\end{equation}
provided $C_4=\frac{C_3}{\gamma\mathrm{e}^{S_0}\mu_0^{\gamma-1}}$.
Integrating this inequality over $[0,r)$ for $r\in[0,\epsilon_0)$ derives that
\begin{equation}
\frac{2r}{\gamma+1}\frac{M_0^2}{M_0^2-1}
(\gamma+1)C_4\frac{(r_2^2)^{\frac{\gamma-1}{\gamma+1}}}{r_1}\leq
(\bar M^2(r))^{\frac{2}{\gamma+1}}<(M_0^2)^{\frac{2}{\gamma+1}}\quad \hbox{for}\quad r\in(0,\epsilon_0).
\end{equation}
This, together with \eqref{mono}, implies that  $\displaystyle\lim_{r\to\epsilon_0}\bar M^2(r)$ exists, and hence $\bar M^2$ can be extended up to $r=\epsilon_0$. By repeating the above argument, the solution $(\bar M^2,\bar E)(r)$ can be extended up to $r=R$. Therefore, the initial value problem \eqref{sreq1}  and   \eqref{-srco-c} has a unique smooth subsonic solution on $[0,R]$.
\end{proof}
Define
\begin{equation*}
\bar\rho(r)=\frac{J_0}{\hat{r}\bar U},\ \ \bar P(r)=\mathrm{e}^{S_0}\bar\rho^{\gamma}, \ \ \bar{\Phi}(r)=-\int_{0}^r\bar E(t)\mathrm{d}t, \ \ r\in[0,R].
\end{equation*}
Then $(\bar\rho,\bar U,\bar P,\bar{\Phi})(r)$ satisfies \eqref{-sreq} in $[0,R]$. Furthermore,  the pseudo-Bernoulli function $\bar{\mathscr{K}}$ is given by
\begin{eqnarray*}
\bar{\mathscr{K}}(r)=\frac{\bar U^2}{2}+\frac{\gamma e^{S_0}\bar\rho^{\gamma-1}}{\gamma-1}-\bar\Phi,\ \  r\in[0, R].
\end{eqnarray*}
It follows from the third equation in \eqref{J-0} that
 \begin{eqnarray}\label{backK}
\bar{\mathscr{K}}(r)={\mathscr{K}}_0, \ \   {\mathscr{K}}_0= \frac{U_0^2}{2}+\frac{\gamma e^{S_0}\rho_0^{\gamma-1}}{\gamma-1}, \ \  r\in[0, R].
 \end{eqnarray}
\begin{definition}\label{defe1}
$(\bar\rho,\bar  U, \bar P, \bar \Phi)$   is called the  background  solution  associated with the entrance data $(b_0, \rho_0,$\\$ U_{0},P_{0}, E_{0})$.
\end{definition}

\subsection{The stability  problem and main results}\noindent
\par This paper aims to establish the structural stability of the  background solution
under perturbations of suitable boundary conditions on the entrance and exit of the convergent nozzle. Note that $(\tilde{\rho},\tilde{U},\tilde{V},\tilde{P},\tilde{\Phi})$ solves \eqref{EP} if and only if it solves
\begin{equation}\label{EP1}
\begin{cases}
\begin{aligned}
&(\tilde{r}\tilde{\rho} \tilde{U})_{\tilde{r}}+(\tilde{\rho} \tilde{V})_{\tilde{\theta}}=0,\\
&\tilde{\rho} \tilde{U}\tilde{V}_{\tilde{r}}+\frac{1}{\tilde{r}}\tilde{\rho} \tilde{U}\tilde{V}+\frac{1}{\tilde{r}}\tilde{\rho} \tilde{V}\tilde{V}_{\tilde{\theta}}+\frac{1}{\tilde{r}}\tilde{P}_{\tilde{\theta}}
=\frac{1}{\tilde{r}}\tilde{\rho}\tilde{\Phi}_{\tilde{\theta}},\\
&\tilde{U}\tilde{S}_{\tilde{r}}+\frac{1}{\tilde{r}}\tilde{V}\tilde{S}_{\tilde{\theta}}=0,\\
&\tilde{U}\tilde{\mathscr{K}}_{\tilde{r}}+\frac{1}{\tilde{r}}\tilde{V}\tilde{\mathscr{K}}_{\tilde{\theta}}=0,\\
&\tilde{\Phi}_{\tilde{r}\tilde{r}}+\frac{1}{\tilde{r}^2}\tilde{\Phi}_{\tilde{\theta}\tilde{\theta}}
+\frac{1}{\tilde{r}}\tilde{\Phi}_{\tilde{r}}=\tilde{\rho}-\tilde b,
\end{aligned}
\end{cases}
\end{equation}
provided $\tilde{\rho}>0$ and $\tilde{U}\neq0$.
  By using the new variable $(r,\hat{r},\theta)=(r_2-\tilde{r},\tilde{r},\tilde{\theta})$ and setting
\begin{equation}
(\rho, U, V, P, \Phi,b)(r,\theta):=(\tilde{\rho}, -\tilde{U}, \tilde{V}, \tilde{P}, \tilde{\Phi},\tilde b)(\tilde{r},\tilde{\theta}),
\end{equation}
  \eqref{EP1} can be transformed into
\begin{equation}\label{rEP1}
\begin{cases}
\begin{aligned}
&(\hat{r}\rho U)_{r}+(\rho V)_{\theta}=0,\\
&\rho UV_{r}-\frac{1}{\hat{r}}\rho UV+\frac{1}{\hat{r}}\rho VV_{\theta}+\frac{1}{\hat{r}}P_{\theta}
=\frac{1}{\hat{r}}\rho\Phi_{\theta},\\
&US_r+\frac{1}{\hat{r}}VS_{\theta}=0,\\
&U\mathscr{K}_r+\frac{1}{\hat{r}} V\mathscr{K}_{\theta}=0,\\
&\Phi_{rr}+\frac{1}{\hat{r}^2}\Phi_{\theta\theta}
-\frac{1}{\hat{r}}\Phi_{r}=\rho-b,
\end{aligned}
\end{cases}
\end{equation}
where the pseudo-Bernoulli function $\mathscr{K}$ is given by
\begin{equation}\label{realK}
\mathscr{K}=\frac{U^2+V^2}{2}+\frac{\gamma \mathrm{e}^{S}\rho^{\gamma-1}}{\gamma-1}-\Phi.
\end{equation}
The two-dimensional convergent nozzle $\tilde\Omega$  becomes
\begin{equation*}
\Omega=\Big\{(r,\theta):0<r<R=r_2-r_1, -\theta_0<\theta<\theta_0\Big\}.
\end{equation*}
The boundary $\partial\Omega$ consists of
\begin{equation*}
\begin{split}
\Gamma_{en}&:=\{(r,\theta):r=0,-\theta_0<\theta<\theta_0\},\\
\Gamma_{ex}&:=\{(r,\theta):r=R,-\theta_0<\theta<\theta_0\},\\
\Gamma_\omega&:=\{(r,\theta):0<r<R,\theta=\pm\theta_0\}.
\end{split}
\end{equation*}
\par The main concern in this paper is to solve the following problem.
\begin{problem}\label{pro1}
Given functions $(V_{en}, \Phi_{en}, \mathscr{K}_{en}, S_{en}, \Phi_{ex}, P_{ex}, b)$ sufficiently close to $(0, \bar{\Phi}(0), \mathscr{K}_{0},S_0,$\\$ \bar\Phi(R), \bar P(R), b_0)$, find a solution $(\rho,U,V,P,\Phi)$ to the system \eqref{rEP1} in $\Omega$, subject to the boundary conditions
\begin{equation}\label{BC}
\begin{cases}
\begin{aligned}
&(V,\mathscr{K},S,\Phi)=(V_{en},\mathscr{K}_{en},S_{en},\Phi_{en}),\ \ &{\rm{on}}\ \ \Gamma_{en},\\
&(P,\Phi)=(P_{ex},\Phi_{ex}),\ \ &{\rm{on}}\ \ \Gamma_{ex},\\
&V=\partial_{\theta}\Phi=0, \ \ &{\rm{on}}\ \ \Gamma_\omega.
\end{aligned}
\end{cases}
\end{equation}
with $U^2+V^2<\frac{\gamma P}{\rho}$, i.e., the flow is subsonic in $\overline{\Omega}$.
\end{problem}
Before we state the  main theorem, we first introduce some H\"{o}lder spaces and their norms.  Set  $ \mathbf{y}=(r,\theta) $ and  $ \hat{\mathbf{y}}=(\hat r,\hat \theta) $.  For  any non-negative  integer $ k$ and $ \alpha\in(0,1) $, we define the standard H\"{o}lder norms by
    \begin{equation*}
    \begin{aligned}
     \|u\|_{k;\overline\Omega}:=\sum_{|\beta|\leq k}\sup_{\mathbf{y}\in \Omega}|D^{\beta}u( {\mathbf{y}})|; \
   \|u\|_{k,\alpha;\overline\Omega}:=\|u\|_{k;\overline\Omega}+ \sum_{|\beta|=k}\sup_{\mathbf{y},\hat {\mathbf{y}}\in \Omega,\mathbf{y}\neq \hat {\mathbf{y}}}
  \frac{|D^{\beta}u(\mathbf{y})-D^{\beta}u(\hat {\mathbf{y}})|}{|\mathbf{y}-\hat {\mathbf{y}}|^{\alpha}},
  \end{aligned}
  \end{equation*}
 where $D^{\beta}$ denotes $\p_{r}^{\beta_1}\p_{\theta}^{\beta_2} $ for  $\beta= (\beta_1,\beta_2)$ with $\beta_j\in \mathbb{Z}_+$ and $|\beta|=\sum_{j=1}^2\beta_j$. $ C^{k,\alpha}(\overline\Omega) $ denotes the completion of the set of all smooth functions whose $\|\cdot\|_{k,\alpha;\overline\Omega}$ norms are finite. Furthermore, for a vector function $ {\bf u}=(u_1,u_2,\cdots,u_n) $, define
\begin{equation*}
   \|{\bf u}\|_{k,\alpha;\overline\Omega}:=\sum_{i=1}^{n}\| u_i\|_{k,\alpha;\overline\Omega}.
\end{equation*}
Next, we introduce the following notation: a constant $C$ is said to depend  only on the background
 data if $C$ is chosen depending only on $  (  \gamma,b_0, \rho_0,U_0,P_0,E_0,r_1, r_2, \theta_0)$.
\begin{theorem}\label{thm1}
Let $(\bar\rho,\bar U,\bar P,\bar{\Phi})$ be the background  subsonic solution defined in Definition \ref{defe1}. For given functions $(V_{en}, \Phi_{en}, \mathscr{K}_{en}, S_{en}, \Phi_{ex}, P_{ex}, b)\in \left(C^{2,\alpha}([-\theta_0,\theta_0])\right)^6\times C^{0,\alpha}(\overline{\Omega})$, set
\begin{equation}\label{sigmar}
\begin{split}
\sigma(V_{en}, \Phi_{en}, \mathscr{K}_{en}, S_{en}, \Phi_{ex}, P_{ex}, b):=&
\Vert b-b_0\Vert_{0,\alpha;\overline{\Omega}}+\Vert (V_{en}, \Phi_{en}, \mathscr{K}_{en}, S_{en}, \Phi_{ex}, P_{ex})\\
&-(0, \bar{\Phi}(0), \mathscr{K}_{0}, S_0, \bar\Phi(R), \bar P(R))\Vert_{2,\alpha;[-\theta_0,\theta_0]}.
\end{split}
\end{equation}
There exist positive constants $\sigma_*$ and $\mathcal{C}_*$ depending only on the background data and $\alpha$ such that if $(V_{en}, \Phi_{en}, \mathscr{K}_{en}, S_{en}, \Phi_{ex}, P_{ex}, b)$ satisfies
\begin{equation}
\sigma(V_{en}, \Phi_{en}, \mathscr{K}_{en}, S_{en}, \Phi_{ex}, P_{ex}, b)\leq\sigma_*
\end{equation}
and the compatibility conditions
\begin{equation}\label{compati}
V_{en}(\pm\theta_0)=\Phi'_{en}(\pm\theta_0)=\mathscr{K}'_{en}(\pm\theta_0)
=S'_{en}(\pm\theta_0)=\Phi'_{ex}(\pm\theta_0)=P'_{ex}(\pm\theta_0)=0,
\end{equation}
 then Problem \ref{pro1} has a unique subsonic solution $(\rho,U,V,P,\Phi)\in \left(C^{2,\alpha}(\overline{\Omega})\right)^5$ satisfying the estimate
\begin{equation}
\Vert (\rho,U,V,P,\Phi)-(\bar\rho,\bar U,0,\bar P,\bar{\Phi})\Vert_{2,\alpha;\overline{\Omega}}
\leq \mathcal{C}_*\sigma(V_{en}, \Phi_{en}, \mathscr{K}_{en}, S_{en}, \Phi_{ex}, P_{ex}, b).
\end{equation}
\end{theorem}
\begin{remark}
{\it The background solution $(\bar\rho,\bar U,\bar P,\bar{\Phi})$ given in Definition \ref{defe1} solves Problem \ref{pro1} with $\sigma_*=0$. Therefore, Problem \ref{pro1} can be regarded as the structural stability problem for the background solution under small perturbations of the boundary data.}
\end{remark}
\begin{remark}
{\it  There are several different decomposition for the steady  Euler-Poisson system in the subsonic region developed by many  researchers from different points of views. In  \cite{Bae2,Bae5},
the authors   utilized the Helmholtz decomposition of the velocity field  to establish the structural stability of  subsonic Euler-Poisson flows with nonzero vorticity in flat nozzles. By the stream function formulation,  the unique existence and stability of subsonic flow for  the steady self-gravitating  Euler-Poisson system
  were obtained in \cite{Bae4,Cao,Duan,WZ25}. However, due to the geometry
structure of the convergent  nozzle,  it turns out to be difficult to obtain  a priori $H^1$ estimate for the associated  elliptic system through the Helmholtz decomposition and the stream function formulation. Here we use the deformation-curl-poisson decomposition developed in \cite{Weng} to
effectively decouple the hyperbolic and elliptic modes in the steady Euler-Poisson system. One of the crucial advantages for this decomposition is to yield a more tractable structure of the linearized  second order elliptic system for the velocity potential
 and the electrostatic potential, which  enables us to establish the well-posedness of the associated nonlinear boundary value problem. Another one is that the solution  obtained in Theorem \ref{thm1} has the same regularity
for the velocity, the density, the pressure and  the electric  potential, which is different from \cite{Bae2,Bae4,Bae5}.}
\end{remark}

The rest of this paper is organized as follows. In Section 2, we first employ the deformation-curl-Poisson decomposition to   reformulate the steady Euler-Poisson system. Then we linearize the nonlinear  system to obtain   a second order linear elliptic system. Furthermore, the well-posedness of the boundary value problem of this system is established via a priori estimates, in which a special structure given by Lemma \ref{lem2} is the key point. In Section 3, we establish the nonlinear structural stability of the background subsonic flow under perturbations of boundary conditions by using the iteration method and the estimates for the linearized elliptic system.

\section{The linearized problem}\noindent
\par In this section, the deformation-curl-Poisson decomposition is utilized   to deal with hyperbolic-elliptic coupled structure in the steady Euler-Poisson system. We introduce a boundary value problem of a second order linearized elliptic system with transport equations for entropy and pseudo-Bernoulli's invariant, subject to the nonlinear boundary conditions for Problem \ref{pro1}. Then a priori estimates are established to obtain the well-posedness of the linearized problem.
\subsection{The deformation-curl-Poisson decomposition}\noindent
\par Firstly, the   pseudo-Bernoulli quantity and  the entropy   are transported by the following equations:
\begin{align}\label{2-1}
 &U\mathscr{K}_r+\frac{1}{\hat{r}} V\mathscr{K}_{\theta}=0, \\\label{2-x}
   &US_r+\frac{1}{\hat{r}}VS_{\theta}=0.
 \end{align}
 Furthermore, it follows from \eqref{realK} that the density can be represented as
\begin{equation}\label{2rho}
\rho=H(S,\mathscr{K},U,V,\Phi)=
\left(\frac{\gamma-1}{\gamma\mathrm{e}^{S}}\left(\mathscr{K}+\Phi-\frac{U^2+V^2}{2}\right)\right)
^{\frac{1}{\gamma-1}}.
\end{equation}
Substituting \eqref{2rho} into the continuity equation in \eqref{rEP1} leads to
\begin{equation}\label{2rho-de}
(\hat{r}H(S,\mathscr{K},\Phi,U,V) U)_{r}+(H(S,\mathscr{K},\Phi,U,V) V)_{\theta}=0.
\end{equation}
In addition, one can follow from the second equation in \eqref{rEP1} to derive that
\begin{equation}\label{2vor}
U((\hat{r}V)_r-U_{\theta})=\frac{\mathrm{e}^{S}H^{\gamma-1}(S,\mathscr{K},U,V,\Phi)}
{\gamma-1}S_{\theta}-\mathscr{K}_{\theta}.
\end{equation}
\par Therefore, if a smooth flow does not contain the vacuum and the stagnation points, then the system \eqref{rEP1} is equivalent to the following system:
\begin{equation}\label{r1EP1}
\begin{cases}
\begin{aligned}
&(\hat{r}H(S,\mathscr{K},\Phi,U,V) U)_{r}+(H(S,\mathscr{K},\Phi,U,V) V)_{\theta}=0,\\
&(\hat{r}\Phi_{r})_r+\frac{1}{\hat{r}}\Phi_{\theta\theta}
=\hat{r}(H(S,\mathscr{K},\Phi,U,V)-b),\\
&U((\hat{r}V)_r-U_{\theta})=\frac{\mathrm{e}^{S}H^{\gamma-1}(S,\mathscr{K},\Phi,U,V)}
{\gamma-1}S_{\theta}-\mathscr{K}_{\theta},\\
&U\mathscr{K}_r+\frac{1}{\hat{r}} V\mathscr{K}_{\theta}=0,\\
&US_r+\frac{1}{\hat{r}}VS_{\theta}=0.\\
\end{aligned}
\end{cases}
\end{equation}
Set
\begin{equation}\label{newva}
\begin{split}
(\mathcal{U},\mathcal{V},\check{\Phi},\mathcal{S},\mathcal{K})(r,\theta)
=(U,V,\Phi,S,\mathscr{K})(r,\theta)-(\bar U(r),0,\bar\Phi(r),S_0,\mathscr{K}_0).
\end{split}
\end{equation}
Define the solution space $\mathcal{J}(\delta)$ which consists of $\mathbf{V}=(\mathcal{U},\mathcal{V},\check{\Phi},\mathcal{S},\mathcal{K})\in \left(C^{2, \alpha}(\overline\Omega)\right)^5$  satisfying the estimate
\begin{equation}\label{Ves}
\Vert{\bf V}\Vert_{2,\alpha;\overline\Omega}:=\Vert (\mathcal{U},\mathcal{V},\check{\Phi},\mathcal{S},\mathcal{K})\Vert_{2, \alpha;\overline\Omega}\leq \delta
\end{equation}
and the compatibility conditions
\begin{equation}\label{Vco}
\mathcal{V}(r,\pm \theta_0)=\partial_\theta(\mathcal{U},\check{\Phi},\mathcal{S},\mathcal{K})(r,\pm \theta_0)=0\ \ \hbox{for} \ \ r\in[0,R],
\end{equation}
where $\delta>0$ to be determined later.
 Given any $\mathbf{V}^{\sharp}=(\mathcal{U}^{\sharp},\mathcal{V}^{\sharp},\check{\Phi}^{\sharp},\mathcal{S}^{\sharp},\mathcal{K}^{\sharp}) \in \mathcal{J}(\delta)$, we will construct an iterative procedure that generates a new $\mathbf{V} \in \mathcal{J}(\delta)$, and thus we define a mapping $\mathcal{T}$ from $\mathcal{J}(\delta)$ to itself by choosing a suitably small positive constant $\delta$.
\par We first solve the transport equations for  $\mathcal{S}$ and $\mathcal{K}$ as follows:
\begin{equation}\label{St}
\begin{cases}
\begin{aligned}
&\bigg(\partial_{r}+\frac{\mathcal{V}^{\sharp}}{\hat{r}(\bar U+\mathcal{U}^{\sharp})}\partial_{\theta} \bigg) \mathcal{S}=0,\\
&\mathcal{S}(0,\theta)=S_{en}(\theta)-S_0,
\end{aligned}
\end{cases}
\end{equation}
and
\begin{equation}\label{Kt}
\begin{cases}
\begin{aligned}
&\bigg(\partial_{r}+\frac{\mathcal{V}^{\sharp}}{\hat{r}(\bar U+\mathcal{U}^{\sharp})}\partial_{\theta} \bigg) \mathcal{K}=0,\\
&\mathcal{K}(0,\theta)=\mathscr{K}_{en}(\theta)-\mathscr{K}_0.
\end{aligned}
\end{cases}
\end{equation}
For any point $(r,\theta)\in \overline\Omega$, the functions $\mathcal{S}$ and $\mathcal{K}$ are conserved along the trajectory determined by the ODE equation
\begin{equation}\label{ode}
\begin{cases}
\begin{aligned}
&\frac{\mathrm{d}}{\mathrm{d}s}\hat \eta(s;r,\theta)=
\frac{\mathcal{V}^{\sharp}}{\hat{r}(\bar U+\mathcal{U}^{\sharp})}(s,\hat \eta(s;r,\theta)),\\
&\hat \eta(r;r,\theta)=\theta.
\end{aligned}
\end{cases}
\end{equation}
Since $\mathbf{V}^{\sharp}\in \mathcal{J}(\delta)$, then we can obtain  that $ \hat\eta(0;r,\theta)\in C^{2,\alpha }(\overline{\Omega})$ satisfies
\begin{equation*}
\|\hat\eta(0;r,\theta) \|_{{2,\alpha };  \overline{\Omega} } \leq C,
\end{equation*}
where $C>0 $ is a constant depending only on the background data and $ \alpha $.
Thus it holds that
\begin{equation}\label{3-7}
\begin{cases}
\begin{aligned}
&\mathcal{S}(r,\theta)=S_{en}(\hat \eta(0;r,\theta))-S_0, \\
&\mathcal{K}(r,\theta)=\mathscr{K}_{en}(\hat \eta(0;r,\theta))-\mathscr{K}_0.
\end{aligned}
 \end{cases}
\end{equation}
One can further verify that $(\mathcal{S},\mathcal{K})$ satisfies following compatibility conditions:
\begin{equation}\label{3-8}
\partial_\theta(\mathcal{S},\mathcal{K})(r,\pm \theta_0)=0\ \ \hbox{for} \ \ r\in[0,R].
\end{equation}
\par Next, it follows from \eqref{r1EP1} that
 the functions $\mathcal{U}$, $\mathcal{V}$ and $\check{\Phi}$ satisfy the following system
\begin{equation}\label{elli1}
\begin{cases}
\begin{aligned}
&\partial_{r}(A_{11}(r)\mathcal{U})+\partial_{\theta}(
A_{22}(r)\mathcal{V})+\partial_{r}(b_1(r)\check{\Phi})=\partial_r f_1+\partial_{\theta} f_2,\\
&\partial_r(\hat{r}\check{\Phi}_{r})+\frac{1}{\hat{r}}\partial_{\theta}^2\check{\Phi}
+c_1(r)\mathcal{U}+c_2(r)\check{\Phi}
=f_3,\\
&\partial_{r}(\hat{r}\mathcal{V})-\partial_{\theta}\mathcal{U}=f_4,
\end{aligned}
\end{cases}
\end{equation}
where
\begin{equation}\label{A0coff}
\begin{split}
A_{11}(r)=&\hat{r}\bar\rho(1-\bar M^2),\quad A_{22}(r)=\hat{r}\bar \rho,\quad
b_1(r)=\frac{\hat{r}\bar\rho\bar U}{\bar c^2},
\quad c_1(r)=\frac{\hat{r}\bar\rho\bar U}{\bar c^2},\quad c_2(r)=-\frac{\hat{r}\bar\rho}{\bar c^2},
\end{split}
\end{equation}
and
\begin{equation}\label{ogif}
\begin{split}
f_1=&-\bigg(\hat{r}H(\mathcal{S}+S_0,\mathcal{K}+\mathscr{K}_0,\mathcal{U}^{\sharp}+\bar U,\mathcal{V}^{\sharp},\check{\Phi}^{\sharp}+\bar\Phi)(\mathcal{U}^{\sharp}+\bar U)-\hat{r} H(S_0,\mathscr{K}_0,\bar U,0,\bar\Phi)\bar U\bigg)\\
&-\hat{r}\bar\rho(1-\bar M^2)\mathcal{U}^{\sharp}-\frac{\hat{r}\bar\rho\bar U}{\bar c^2}\check{\Phi}^{\sharp},\\
f_2=&-\hat{r}H(\mathcal{S}+S_0, \mathcal{K}+\mathscr{K}_0,\mathcal{U}^{\sharp}+\bar U,\mathcal{V}^{\sharp},\check{\Phi}^{\sharp}+\bar\Phi)\mathcal{V}^{\sharp}
-\hat{r}\bar\rho\mathcal{V}^{\sharp},\\
f_3=&\hat{r}(H(\mathcal{S}+S_0, \mathcal{K}+\mathscr{K}_0,\mathcal{U}^{\sharp}+\bar U,\mathcal{V}^{\sharp},\check{\Phi}^{\sharp}+\bar\Phi)-H(S_0,\mathscr{K}_0,\bar U,0,\bar\Phi)-(b-b_0))\\
& -\frac{\hat{r}\bar\rho\bar U}{\bar c^2}\mathcal{U}^{\sharp}+\frac{\hat{r}\bar\rho}{\bar c^2}\check{\Phi}^{\sharp},\\
f_4=&\frac{1}{\mathcal{U}^{\sharp}+\bar U}\bigg(\frac{\mathrm{e}^{\mathcal{S}+S_0}H^{\gamma-1}(\mathcal{S}+S_0,\mathcal{K}+\mathscr{K}_0,
\mathcal{U}^{\sharp}+\bar U,\mathcal{V}^{\sharp},\check{\Phi}^{\sharp}+\bar\Phi)}{\gamma-1}
\mathcal{S}_{\theta}-\mathcal{K}_{\theta}\bigg).
\end{split}
\end{equation}
\par Now we compute the boundary conditions $(\mathcal{U},\mathcal{V},\check{\Phi})$ corresponding to \eqref{BC}. Owing to \eqref{realK}, \eqref{backK}, one obtains
\begin{equation}\label{realK1}
U^2+V^2=2\left(\mathscr{K}+\Phi-\frac{\gamma \mathrm{e}^{S}\rho^{\gamma-1}}{\gamma-1}\right).
\end{equation}
This together with \eqref{newva} yields
\begin{equation}
\begin{split}
U^2+V^2=(\mathcal{U}+\bar U)^2+\mathcal{V}^2
=\mathcal{U}^2+2\bar U\mathcal{U}+\bar U^2+\mathcal{V}^2.
\end{split}
\end{equation}
Direct calculations deduce that
\begin{equation*}
\begin{split}
\mathcal{U}=\frac{1}{\bar U}\left(\mathscr{K}+\Phi-\frac{\gamma\mathrm{e}^{S}\rho^{\gamma-1}}{\gamma-1}\right)
-\frac{1}{\bar U}\left(\mathscr{K}_0+\bar\Phi-\frac{\gamma\mathrm{e}^{S_0}\bar\rho^{\gamma-1}}{\gamma-1}
\right)
-\frac{\mathcal{U}^2+\mathcal{V}^2}{2\bar U}.
\end{split}
\end{equation*}
Thus there holds
\begin{equation}\label{ogih}
\begin{split}
\mathcal{U}=h(\theta;\mathcal{U}^{\sharp},\mathcal{V}^{\sharp},\mathcal{K},\mathcal{S}), \ \ {\rm{on}}\ \ \Gamma_{ex},\\
\end{split}
\end{equation}
where
\begin{equation*}
\begin{split}
h(\theta;\mathcal{U}^{\sharp},\mathcal{V}^{\sharp},\mathcal{K},\mathcal{S})&=\frac{1}{\bar U(R)}\bigg(
\mathcal{K}(R,\theta)+\bigg(
\Phi_{ex}(\theta)-\frac{\gamma}{\gamma-1}(\mathrm{e}^{S_0+\mathcal{S}(R,\theta)})
^{\frac{1}{\gamma}}P_{ex}^{1-\frac{1}{\gamma}}(\theta)\bigg)\\
&\ \ -\bigg(\bar\Phi(R)-\frac{\gamma}{\gamma-1}(\mathrm{e}^{S_0})^{\frac{1}{\gamma}}\bar P^{1-\frac{1}{\gamma}}(R)\bigg)\bigg)
 -\frac{\left((\mathcal{U}^{\sharp})^2+(\mathcal{V}^{\sharp})^2\right)(R,\theta)}
{2\bar U(R)}.
\end{split}
\end{equation*}
Hence the boundary conditions for $(\mathcal{U},\mathcal{V},\check{\Phi})$ are transformed into
\begin{equation}\label{aBC1}
\begin{cases}
\begin{aligned}
&(\mathcal{V},\check{\Phi})=(V_{en},\Phi_{en}-\bar \Phi(0)),\ \ &{\rm{on}}\ \ \Gamma_{en},\\
&(\mathcal{U},\check{\Phi})=(h,\Phi_{ex}-\bar \Phi(R)),\ \ &{\rm{on}}\ \ \Gamma_{ex},\\
&{\mathcal{V}=\partial_{\theta}\check{\Phi}=0}, \ \ &{\rm{on}}\ \ \Gamma_\omega.
\end{aligned}
\end{cases}
\end{equation}
 Use the abbreviation
    \begin{equation*}
    \sigma_p=\sigma(V_{en}, \Phi_{en}, \mathscr{K}_{en}, S_{en}, \Phi_{ex}, P_{ex}, b),
    \end{equation*}
    where $\sigma(V_{en}, \Phi_{en}, \mathscr{K}_{en}, S_{en}, \Phi_{ex}, P_{ex}, b)$  is  defined in \eqref{sigmar}. Then a direct computation shows that
\begin{equation}\label{aBC1-e}
\Vert (f_1,f_2)\Vert_{2,\alpha;\overline\Omega}
+\Vert f_3\Vert_{0,\alpha;\overline\Omega}
+\Vert f_4\Vert_{1,\alpha;\overline\Omega}
+\Vert h\Vert_{2,\alpha;[-\theta_0,\theta_0]} \leq C(\delta^2+\sigma_p),
\end{equation}
and
\begin{equation}\label{compatibilityf}
f_2(r,\pm \theta_0)=\partial_{\theta}^2f_2(r,\pm \theta_0)=\partial_{\theta}f_1(r,\pm \theta_0)=f_4(r,\pm \theta_0)=h'(\pm \theta_0)=0 \ \ \hbox{for} \ \ r\in[0,R].
\end{equation}
Here  $C>0 $ is a constant depending only on the background data and $ \alpha $.
\par In the following, we first consider the following problem
\begin{equation}\label{extend}
\begin{cases}
\begin{aligned}
&\partial_r(\hat{r}\phi_r)+\partial_\theta^2\phi=f_4,\\
&\phi_r(0,\theta)=0, \ \
\phi(R,\theta)=0, \\
&\phi(r,\pm \theta_0)=0.
\end{aligned}
\end{cases}
\end{equation}
To deal with the singularity near the corner, one can use the standard symmetric extension technique to extend $\phi$ and $f_4$ as
\begin{equation}\label{4-26}
(\phi^e,f_4^e)(r,\theta)=
\begin{cases}
(\phi,f_4)(r,\theta), \ \ &{\rm{for}} \ \ (r,\theta)\in [0,R]\times[-\theta_0,\theta_0],\\
-(\phi,f_4)(r,-\theta), \ \ &{\rm{for}} \ \ (r,\theta)\in [0,R]\times(-3\theta_0,-\theta_0),\\
-(\phi,f_4)(r,2\theta_0-\theta), \ \ &{\rm{for}} \ \ (r,\theta)\in [0,R]\times(\theta_0,3\theta_0).
\end{cases}
\end{equation}
Then $\phi^e$ satisfies
\begin{equation}\label{extend1}
\begin{cases}
\begin{aligned}
&\partial_r(\hat{r}\partial_r\phi^e)+\partial_\theta^2\phi^e=f_4^e,\\
&\partial_r\phi^e(0,\theta)=\phi^e(R,\theta)=0,\\
&\phi^e(r,-3\theta_0)=\phi^e(r,3\theta_0)=0.
\end{aligned}
\end{cases}
\end{equation}
Consequently, the regularity of $\phi$  can be improved to be $C^{3,\alpha}(\overline\Omega)$ with the estimate
\begin{equation}\label{regu}
\Vert\phi\Vert_{3,\alpha;\overline\Omega}\leq C\Vert f_4\Vert_{1, \alpha;\overline\Omega},
\end{equation}
provided $C$ is a constant. Define
\begin{equation}\label{transUV1}
\check{\mathcal{U}}=\mathcal{U}+ \phi_{\theta},\quad
\check{\mathcal{V}}=\mathcal{V}- \phi_r.
\end{equation}
Then \eqref{elli1} and \eqref{aBC1} are transformed into
\begin{equation}\label{elli2}
\begin{cases}
\begin{aligned}
&\partial_{r}(A_{11}(r)\check{\mathcal{U}}
)+\partial_{\theta}(A_{22}(r)\check{\mathcal{V}})
+\partial_{r}(b_1(r)\check{\Phi})=\partial_r \mathfrak{f}_1+\partial_{\theta} \mathfrak{f}_2,\\
&\partial_r(\hat{r}\check{\Phi}_{r})+\frac{1}{\hat{r}}\partial_{\theta}^2\check{\Phi}
+c_1(r)\check{\mathcal{U}}
+c_2(r)\check{\Phi}=\mathfrak{f}_3,\\
&\partial_{r}(\hat{r}\check{\mathcal{V}})-\partial_{\theta}\check{\mathcal{U}}=0,
\end{aligned}
\end{cases}
\end{equation}
with the boundary conditions
\begin{equation}\label{aBC2}
\begin{cases}
\begin{aligned}
&(\check{\mathcal{V}},\check{\Phi})=(V_{en},\Phi_{en}-\bar \Phi(0)),\ \ &{\rm{on}}\ \ \Gamma_{en},\\
&(\check{\mathcal{U}},\check{\Phi})=({h},\Phi_{ex}-\bar \Phi(R)),\ \ &{\rm{on}}\ \ \Gamma_{ex},\\
&{\check{\mathcal{V}}=\partial_{\theta}\check{\Phi}=0}, \ \ &{\rm{on}}\ \ \Gamma_\omega,
\end{aligned}
\end{cases}
\end{equation}
where
\begin{equation}\label{transf1}
\begin{split}
\mathfrak{f}_1=f_1+A_{11}(r)\partial_{\theta} \phi, \ \
\mathfrak{f}_2=f_2-A_{22}(r)\partial_{r} \phi,\ \
\mathfrak{f}_3=f_3+c_1(r)\partial_{\theta} \phi.\\
\end{split}
\end{equation}
The third equation in \eqref{elli2} implies that there exists a potential function $\psi$ such that
\begin{equation}\label{transUV2}
\partial_{\theta}\psi=\hat{r}\check{\mathcal{V}},\quad \partial_{r}\psi=\check{\mathcal{U}},
\quad \psi(0,-\theta_0)=0.
\end{equation}
Then \eqref{elli2} and \eqref{aBC2} become
\begin{equation}\label{elli3}
\begin{cases}
\begin{aligned}
&L_1(\psi,\check{\Phi})=\partial_r \mathfrak{f}_1+\partial_{\theta} \mathfrak{f}_2,\\
&L_2(\psi,\check{\Phi})=\mathfrak{f}_3,
\end{aligned}
\end{cases}
\end{equation}
and
\begin{equation}\label{aBC3}
\begin{cases}
\begin{aligned}
&(\psi,\check{\Phi})=(\mathfrak{g},\Phi_{en}-\bar \Phi(0)),\ \ &{\rm{on}}\ \ \Gamma_{en},\\
&(\partial_{r}\psi,\check{\Phi})=({h},\Phi_{ex}-\bar \Phi(R)),\ \ &{\rm{on}}\ \ \Gamma_{ex},\\
&{\partial_{\theta}\psi=\partial_{\theta}\check{\Phi}=0}, \ \ &{\rm{on}}\ \ \Gamma_\omega,
\end{aligned}
\end{cases}
\end{equation}
provided
\begin{equation*}
\begin{split}
L_1(\psi,\check{\Phi})=&\partial_{r}(a_{11}(r)\partial_{r}\psi
)
+\partial_{\theta}(a_{22}(r)\partial_{\theta}\psi)
+\partial_{r}(b_1(r)\check{\Phi}),\\
L_2(\psi,\check{\Phi})=&\partial_r(\hat{r}\partial_r\check{\Phi})
+\frac{1}{\hat{r}}\partial_{\theta}^2\check{\Phi}+c_1(r)\partial_{r}\psi
+c_2(r)\check{\Phi},
\end{split}
\end{equation*}
and
\begin{equation}\label{Acoff}
a_{11}(r)=\hat{r}\bar\rho(1-\bar M^2),\quad a_{22}(r)=\bar \rho, \ \ \mathfrak{g}(\theta)=\int_{-\theta_0}^{\theta}r_2V_{en}(s)\mathrm{d}s.
\end{equation}
\par To homogenize the boundary conditions, define
\begin{equation}\label{newva1}
\begin{split}
\varphi=\psi-rh-\mathfrak{g},\quad \Psi=\check{\Phi}-\Phi^*,
\end{split}
\end{equation}
and
\begin{equation}
\begin{split}
\Phi^*=\frac{R-r}{R}(\Phi_{en}(\theta)-\bar \Phi(0))+ \frac{r}{R}(\Phi_{ex}(\theta)-\bar \Phi(R)).
\end{split}
\end{equation}
Then $(\psi,\check{\Phi})$ solves \eqref{elli3} and \eqref{aBC3} if and only if $(\varphi,\Psi)$ solves the equations
\begin{equation}\label{elli4}
\begin{cases}
\begin{aligned}
&\mathcal{L}_1(\varphi,\Psi)=\partial_r F_1+\partial_{\theta} F_2,\\
&\mathcal{L}_2(\varphi,\Psi)=F_3,
\end{aligned}
\end{cases}
\end{equation}
with the boundary conditions
\begin{equation}\label{aBC4}
\begin{cases}
\begin{aligned}
&(\varphi,\Psi)=(0,0),\ \ &{\rm{on}}\ \ \Gamma_{en},\\
&(\partial_{r}\varphi,\Psi)=(0,0),\ \ &{\rm{on}}\ \ \Gamma_{ex},\\
&{\partial_{\theta}\varphi=\partial_{\theta}\Psi=0}, \ \ &{\rm{on}}\ \ \Gamma_\omega,
\end{aligned}
\end{cases}
\end{equation}
where
\begin{equation*}
\begin{split}
\mathcal{L}_1(\varphi,\Psi)=&\partial_{r}(a_{11}(r)\partial_{r}\varphi
)+\partial_{\theta}(a_{22}(r)\partial_{\theta}\varphi)
+\partial_{r}(b_1(r)\Psi),\\
\mathcal{L}_2(\varphi,\Psi)=&\partial_r(\hat{r}\partial_r\Psi)
+\frac{1}{\hat{r}}\partial_{\theta}^2\Psi+c_1(r)\partial_{r}\varphi
+c_2(r)\Psi,
\end{split}
\end{equation*}
and
\begin{equation}\label{transf3}
\begin{split}
F_1=&\mathfrak{f}_1-a_{11}h-a_3\Phi^*,\\
F_2=&\mathfrak{f}_2-a_{22}(rh'+\mathfrak{g}'),\\
F_3=&\mathfrak{f}_3-\partial_r(\hat{r}\partial_r\Phi^*)-\frac{1}{\hat{r}}\partial_{\theta}^2\Phi^*
-c_1h-c_2\Phi^*.
\end{split}
\end{equation}
Then it follows from \eqref{aBC1-e} that
\begin{equation}\label{aBC1-e-w}
\begin{aligned}
&\Vert F_1\Vert_{2,\alpha;\overline\Omega}
+\Vert F_2\Vert_{1,\alpha;\overline\Omega}
+\Vert F_3\Vert_{0,\alpha;\overline\Omega}\\
&\leq C\left(\Vert (f_1,f_2)\Vert_{2,\alpha;\overline\Omega}
+\Vert f_3\Vert_{0,\alpha;\overline\Omega}
+\Vert f_4\Vert_{1,\alpha;\overline\Omega}+\sigma_p\right)
 \leq C(\delta^2+\sigma_p),
 \end{aligned}
\end{equation}
where $ C>0$ is a constant depending only on the the background data and $\alpha$.
\begin{lemma}\label{lem2}
Let $(\bar\rho,\bar U,\bar P,\bar{\Phi})$ be the background  subsonic solution defined in Definition \ref{defe1}. Then the associated coefficients $(a_{11},a_{22},b_1,c_1,c_2)(r)$  satisfy the following properties.
\begin{itemize}
\item[(i)] For each $k\in\mathbb{Z}^+$, there exists a constant $C_k>0$ depending only on the background data and $k$ such that
\begin{equation*}
\Vert (a_{11},a_{22},b_1,c_1,c_2)\Vert_{k;\overline\Omega}\leq C_k.
\end{equation*}
\item[(ii)] There exists a positive constant $\lambda_0$ depending only on the background data such that
\begin{equation*}
\lambda_0\vert\bm{\xi}\vert^2\leq \sum_{i=1}^2a_{ii}(r)\xi_i^2
\leq\frac{1}{\lambda_0}\vert\bm{\xi}\vert^2 \ \ \hbox{for all} \ \ (r,\theta)\in\overline{\Omega}\ \ \hbox{and}
\ \ \bm{\xi}=(\xi_1,\xi_2)\in\mathbb{R}^2.
\end{equation*}
\item[(iii)] It holds that $b_1=c_1$ in $\overline\Omega$.
\end{itemize}
\end{lemma}
\begin{remark}
{\it Lemma \ref{lem2} identifies a structural property of the elliptic system derived from the deformation-curl-Poisson decomposition, which is essential for obtaining a priori $H^1$ estimate of the associated linearized problem.}
\end{remark}
\begin{remark}
{\it It is worth mentioning that the Helmholtz decomposition for the velocity field $\mathbf{u}=\nabla \varphi+\nabla^{\perp}\psi$ with $\perp=(\partial_{x_2},-\partial_{x_1})$ in \cite{Bae2} requires a more delicate and technically involved treatment to ensure the unique solvability of the vector potential $\varphi $ and scalar potential $\psi$ in polar coordinates. In contrast, the deformation-curl-Poisson decomposition uses the Bernoulli law to reformulate the Euler-Poisson system by effectively reducing the dimension of the velocity. This leads to a significant simplification of the algebraic structure of the associated elliptic system.}
\end{remark}

\subsection{A priori estimates for the linearized system}\noindent
\par Define a Hilbert space $\mathcal{R}$ as
\begin{equation*}
\mathcal{R}=\{(\xi,\omega)\in[H^1(\Omega)]^2: \xi=\omega=0 \,\, \hbox{on} \,\, \Gamma_{en},\,\omega=0 \,\, \hbox{on} \,\, \Gamma_{ex}\}.
\end{equation*}
\begin{definition}\label{weakk}
The function $(\varphi,\Psi)$ is called a weak solution to the problem \eqref{elli4} and \eqref{aBC4}, if $(\varphi,\Psi)\in\mathcal{R}$ satisfies
\begin{equation}\label{weak}
\mathcal{L}[(\varphi,\Psi),(\xi,\omega)]
=\langle(F_1,F_2,F_3),(\xi,\omega)\rangle
\end{equation}
for all $(\xi,\omega)\in\mathcal{R}$, where
\begin{equation}\label{w1}
\begin{split}
\mathcal{L}[(\varphi,\Psi),(\xi,\omega)]
=\iint_{\Omega}&\bigg(a_{11}\varphi_r\xi_r+a_{22}\varphi_{\theta}\xi_{\theta}
+b_1\Psi\xi_r\\
&+\hat{r}\Psi_r\omega_r+\frac{1}{\hat{r}}\Psi_{\theta}\omega_{\theta}
-c_1\varphi_r\omega-c_2\Psi\omega\bigg)\mathrm{d}r\mathrm{d}\theta,
\end{split}
\end{equation}
and
\begin{equation}\label{ho}
\begin{split}
\langle(F_1,F_2,F_3),(\xi,\omega)\rangle
=\iint_{\Omega}(F_1\xi_r+F_2\xi_{\theta}-F_3\omega)\mathrm{d}r\mathrm{d}\theta
-\int_{\Gamma_{ex}}F_1\xi\mathrm{d}\theta
-\int_{\Gamma_{\omega}}F_2\xi\mathrm{d}r.
\end{split}
\end{equation}
\end{definition}
\begin{proposition}\label{prop1}
Let $(\bar\rho,\bar U,\bar P,\bar{\Phi})$ be the background  subsonic solution defined in Definition \ref{defe1}. There exists a unique weak solution $(\varphi,\Psi)\in\mathcal{R}$ to the problem \eqref{elli4} and \eqref{aBC4}. Furthermore, $(\varphi,\Psi)$ satisfies the estimate
\begin{equation}\label{H^1}
\Vert(\varphi,\Psi)\Vert_{H^1(\Omega)}
\leq C\Vert(F_1,F_2,F_3)\Vert_{0;\overline\Omega}
\end{equation}
with a constant $C$ depending only on the background data.
\end{proposition}
\begin{proof}
It follows from Lemma \ref{lem2} that $b_1=c_1$. This, together with \eqref{A0coff}, yields
\begin{equation}
\begin{split}
\mathcal{L}[(\xi,\omega),(\xi,\omega)]
&=\iint_{\Omega}\bigg(a_{11}\xi_r^2+a_{22}\xi_{\theta}^2+\hat{r}\omega_r^2+\frac{1}{\hat{r}}\omega_{\theta}^2-c_2\omega^2
\bigg)\mathrm{d}r\mathrm{d}\theta\\
&\geq\iint_{\Omega}\bigg(\lambda_0(\xi_{r}^2+\xi_{\theta}^2)+\hat{r}\omega_r^2+
\frac{1}{\hat{r}}\omega_{\theta}^2
+\frac{\hat{r}}{\gamma\mathrm{e}^{S_0}\bar\rho^{\gamma-2}}\omega^2\bigg)\mathrm{d}r\mathrm{d}\theta
\end{split}
\end{equation}
for all $(\xi,\omega)\in\mathcal{R}$ and $(r,\theta)\in\Omega$. Employing Poincar$\mathrm{\acute{e}}$ inequality deduces that
\begin{equation*}
\mathcal{L}[(\xi,\omega),(\xi,\omega)]\geq C_1\Vert(\xi,\omega)\Vert^2_{H^1(\Omega)}
\end{equation*}
with a constant $C_1$ depending on the background data. In addition, using  H$\mathrm{\ddot{o}}$lder inequality, trace inequality and Poincar$\mathrm{\acute{e}}$ inequality gives
\begin{equation*}
\vert\langle(F_1,F_2,F_3),(\xi,\omega)\rangle\vert\leq C_2\Vert(\xi,\omega)\Vert_{H^1(\Omega)}
\Vert(F_1,F_2,F_3)\Vert_{0;\overline\Omega},
\end{equation*}
where the constant $C_2>0$ is independent on the background data. Lax-Milgram theorem implies that, for any $(F_1,F_2,F_3)
\in\left(C^0(\overline\Omega)\right)^3$, there exists a unique $(\varphi,\Psi)\in\mathcal{R}$ such that \eqref{weak} holds. Therefore, the problem \eqref{elli4} and \eqref{aBC4} has a unique weak solution $(\varphi,\Psi)\in\mathcal{R}$ satisfying the estimate \eqref{H^1}. This completes the proof of Proposition \ref{prop1}.
\end{proof}
\begin{proposition}\label{prop2}
Suppose that $(\varphi,\Psi)\in\mathcal{R}$ is the weak solution to \eqref{elli4} and \eqref{aBC4} given by Proposition \ref{prop1}. There exists $\bar{\alpha}\in(0,1)$ depending on the background data such that, for any $\alpha\in(0,\bar{\alpha})$, $(\varphi,\Psi)$ satisfies
\begin{equation}\label{C^c}
\Vert(\varphi,\Psi)\Vert_{0,\alpha;\overline\Omega}
\leq \check{C}\Vert(F_1,F_2,F_3)\Vert_{0;\overline\Omega},
\end{equation}
where $\check{C}>0$ is a constant depending on the background data and $\alpha$.
\end{proposition}
\begin{proof}
For the domain $\Omega$, there exists a constant $R_0>0$ such that, for any $\mathrm{x}^*:=(r^*,\theta^*)\in\Omega$ and $\mathrm{r}\in(0,R_0]$, it holds that
\begin{equation*}
\frac{1}{k_0}\leq\frac{\mathrm{vol}(B_{\mathrm{r}}(\mathrm{x}^*)\cap\Omega)}
{\mathrm{vol}(B_{\mathrm{r}}(\mathrm{x}^*))}\leq k_0
\end{equation*}
with a uniform constant $k_0>0$. It follows from Theorem 3.1 and Corollary 3.2 in \cite{Han} that, if there are constants $C$ and $R_1\leq R_0$ satisfying
\begin{equation}\label{TC}
\iint_{B_{\mathrm{r}}(\mathrm{x}^*)\cap\Omega}\left(\vert\nabla\varphi\vert^2
+\nabla\Psi\vert^2\right)\mathrm{d}r\mathrm{d}\theta
\leq C\mathrm{r}^{2\alpha}\Vert(F_1,F_2,F_3)\Vert_{0;\overline\Omega}^2
\end{equation}
for all $\mathrm{x}^*\in\Omega$ and $\mathrm{r}\in(0,R_1]$, then
\begin{equation*}
\Vert\varphi\Vert_{0,\alpha;\overline\Omega}+\Vert\Psi\Vert_{0,\alpha,\overline\Omega}
\leq C\left(\Vert(F_1,F_2,F_3)\Vert_{0;\overline\Omega}
+\Vert\varphi\Vert_{L^2(\Omega)}+\Vert\Psi\Vert_{L^2(\Omega)}\right).
\end{equation*}
This, together with Proposition \ref{prop1}, gives
\begin{equation*}
\Vert\varphi\Vert_{0,\alpha;\overline\Omega}+\Vert\Psi\Vert_{0,\alpha;\overline\Omega}
\leq C\Vert(F_1,F_2,F_3)\Vert_{0;\overline\Omega},
\end{equation*}
where the constant $C$ depends on the background data and $\alpha$. Hence, \eqref{C^c} is obtained. In what follows, we prove the inequality \eqref{TC}. There exist three cases of $\mathrm{x}^*=(r^*,\theta^*)$:
\begin{itemize}
\item[(i)] $\mathrm{x}^*\in\Omega\setminus\Gamma_b$ and $B_{\mathrm{r}}(\mathrm{x}^*)\subset\Omega$.
\item[(ii)] $\mathrm{x}^*\in\Gamma_p$, where $\Gamma_p=\overline{\Gamma}_{en}\cup\overline{\Gamma}_{ex}\cap\overline{\Gamma}_{\omega}$.
\item[(iii)] $\mathrm{x}^*\in(\Gamma_b\setminus\Gamma_p)$, $B_{\mathrm{r}}(\mathrm{x}^*)\cap\Gamma_p=\emptyset$, where $\Gamma_b=\overline{\Gamma}_{en}\cup\overline{\Gamma}_{ex}\cup\overline{\Gamma}_{\omega}$.
\end{itemize}
Case $(i)$, case $(iii)$ and other general cases can be treated via case $(ii)$. Here we treat only case $(ii)$. Without loss of generality, we fix $\mathrm{x}_0=(r_0,\theta_0)\in\Gamma_{p}\cap\Gamma_{ex}$ provided $r_0=R$. For a constant $\mathrm{r}\in(0,\frac{1}{10}\min\{1,R\})$, set $D_{\mathrm{r}}(\mathrm{x}_0):=B_{\mathrm{r}}(\mathrm{x}_0)\cap\Omega$. Let $\omega_1$ and $\omega_2$ be weak solutions to
\begin{equation}
\begin{cases}
\partial_{r}(a_{11}(\mathrm{x}_0)\partial_r\omega_{1})
+\partial_{\theta}(a_{22}(\mathrm{x}_0)\partial_{\theta}\omega_{1})=0,
&\hbox{in}\quad D_{\mathrm{r}}(\mathrm{x}_0),\\
\omega_{1}=0, \,\,\,\,&\hbox{on}\quad \partial D_{\mathrm{r}}(\mathrm{x}_0)\cap\Gamma_{ex},\\
\partial_{\theta}\omega_{1}=0, \,\,\,\,&\hbox{on}\quad \partial D_{\mathrm{r}}(\mathrm{x}_0)\cap\Gamma_{\omega},\\
\omega_1=\varphi, &\hbox{on}\quad \partial D_{\mathrm{r}}(\mathrm{x}_0)\cap\Omega,
\end{cases}
\end{equation}
and
\begin{equation}
\begin{cases}
\partial_r(\hat{r}\partial_{r}\omega_{2})+\frac{1}{\hat{r}}\partial_{\theta}^2\omega_{2}
=0, &\hbox{in}\quad D_{\mathrm{r}}(\mathrm{x}_0),\\
\omega_2=0, &\hbox{on}\quad \partial D_{\mathrm{r}}(\mathrm{x}_0)\cap\Gamma_{ex},\\
\partial_{\theta}\omega_{2}=0, &\hbox{on}\quad \partial D_{\mathrm{r}}(\mathrm{x}_0)\cap\Gamma_{\omega},\\
\omega_2=\mathcal{W}, &\hbox{on}\quad \partial D_{\mathrm{r}}(\mathrm{x}_0)\cap\Omega.
\end{cases}
\end{equation}
For any $z=(z_1,z_2)\in H^1(D_{\mathrm{r}}(\mathrm{x}_0))\cap\{z=0 \,\, \hbox{on} \,\, \partial D_{\mathrm{r}}(\mathrm{x}_0)\cap(\Omega\cup\Gamma_{ex})\}$, one has
\begin{equation}\label{weak1}
\iint_{D_{\mathrm{r}}(\mathrm{x}_0)}\bigg(a_{11}(\mathrm{x}_0)\partial_r\omega_1\partial_rz_1
+a_{22}(\mathrm{x}_0)\partial_{\theta}\omega_{1}\partial_{\theta}z_{1}
+\hat{r}\partial_{r}\omega_{2}\partial_{r}z_{2}
+\frac{1}{\hat{r}}\partial_{\theta}\omega_{2}\partial_{\theta}z_{2}
\bigg)\mathrm{d}r\mathrm{d}\theta=0.
\end{equation}
We extend $z_1$ and $z_2$ by setting $z_1=z_2=0$ in $\Omega\setminus D_{\mathrm{r}}(\mathrm{x}_0)$. Substituting $(\xi,\omega)=(z_1,z_2)$ into \eqref{weak} and then subtracting \eqref{weak1} from it yield
\begin{equation}\label{weak2}
\begin{split}
\iint_{D_{\mathrm{r}}(\mathrm{x}_0)}\mathcal{T}\mathrm{d}r\mathrm{d}\theta
=\langle(F_1,F_2,F_3),(z_1,z_2)\rangle
-\iint_{D_{\mathrm{r}}(\mathrm{x}_0)}(\mathcal{J}_1+\mathcal{J}_2+\mathcal{J}_3)
\mathrm{d}r\mathrm{d}\theta,
\end{split}
\end{equation}
where
\begin{equation*}
\begin{split}
&\mathcal{T}=a_{11}(\mathrm{x}_0)\partial_r(\varphi-\omega_1)\partial_rz_1
+a_{22}(\mathrm{x}_0)\partial_{\theta}(\varphi-\omega_1)\partial_{\theta}z_1\\
&\qquad+\hat{r}\partial_{r}(\Psi-\omega_2)\partial_{r}z_2
+\frac{1}{\hat{r}}\partial_{\theta}(\Psi-\omega_2)\partial_{\theta}z_2,\\
&\langle(F_1,F_2,F_3),(z_1,z_2)\rangle
=\iint_{\Omega}(F_1\partial_rz_1+F_2\partial_{\theta}z_1-F_3z_2)\mathrm{d}r\mathrm{d}\theta
-\int_{\Gamma_{ex}}F_1z_1\mathrm{d}\theta
-\int_{\Gamma_{\omega}}F_2z_1\mathrm{d}r,\\
&\mathcal{J}_1=
(a_{11}-a_{11}(\mathrm{x}_0))\partial_r\varphi\partial_rz_1
+(a_{22}-a_{22}(\mathrm{x}_0))\partial_{\theta}\varphi\partial_{\theta}z_1,\ \
\mathcal{J}_2=
b_1\Psi\partial_{r}z_1,\ \ \mathcal{J}_3=
-c_1\partial_{r}\varphi z_2
-c_2\Psi z_2.
\end{split}
\end{equation*}
\par Set $(Y_1,Y_2)=(\varphi-\omega_1,\Psi-\omega_2)$ and substitute $(z_1,z_2)=(Y_1,Y_2)$ into \eqref{weak2}. Lemma \ref{lem2} implies
\begin{equation}\label{x1}
\begin{split}
&\iint_{D_{\mathrm{r}}(\mathrm{x}_0)}
\left(a_{11}(\mathrm{x}_0)\vert\partial_rY_1\vert^2
+a_{22}(\mathrm{x}_0)\vert\partial_{\theta}Y_1\vert^2
+\hat{r}\vert\partial_{r}Y_2\vert^2
+\frac{1}{\hat{r}}\vert\partial_{\theta}Y_2\vert^2\right)
\mathrm{d}r\mathrm{d}\theta\\
&\geq\min\{\lambda_0,r_1,\frac{1}{r_2}\}
\iint_{D_{\mathrm{r}}(\mathrm{x}_0)}
\left(\vert\nabla Y_1\vert^2+\vert\nabla Y_2\vert^2\right)
\mathrm{d}r\mathrm{d}\theta.
\end{split}
\end{equation}
Using Lemma \ref{lem2} and H$\mathrm{\ddot{o}}$lder inequality yield
\begin{equation}\label{x2}
\begin{split}
\iint_{D_{\mathrm{r}}(\mathrm{x}_0)}\mathcal{J}_1\mathrm{d}r\mathrm{d}\theta&
\leq C_1
\mathrm{r}\bigg(\iint_{D_{\mathrm{r}}(\mathrm{x}_0)}
\vert\nabla \varphi\vert^2\mathrm{d}r\mathrm{d}\theta\bigg)^{\frac{1}{2}}
\bigg(\iint_{D_{\mathrm{r}}(\mathrm{x}_0)}
\vert\nabla Y_1\vert^2\mathrm{d}r\mathrm{d}\theta\bigg)^{\frac{1}{2}}\\
&\leq C_1\mathrm{r}
\Vert\varphi\Vert_{H^1(\Omega)}
\bigg(\iint_{D_{\mathrm{r}}(\mathrm{x}_0)}
\vert\nabla Y_1\vert^2\mathrm{d}r\mathrm{d}\theta\bigg)^{\frac{1}{2}},
\end{split}
\end{equation}
where $C_1=\Vert a_{11}\Vert_{C^1(\overline{\Omega})}+\Vert a_{22}\Vert_{C^1(\overline{\Omega})}$. Since $\Psi\in H^1(D_{\mathrm{r}}(\mathrm{x}_0))$ with $\Psi=0$ on $\partial D_{\mathrm{r}}(\mathrm{x}_0)\cap\Gamma_{en}$, that is, $\Psi(0,\theta)=0$. Then it holds that
\begin{equation}
\begin{split}
\iint_{D_{\mathrm{r}}(\mathrm{x}_0)}\Psi^2\mathrm{d}r\mathrm{d}\theta&
\leq \frac{1}{2}\mathrm{r}^{2}
\iint_{D_{\mathrm{r}}(\mathrm{x}_0)}\Psi_r^2\mathrm{d}r\mathrm{d}\theta.
\end{split}
\end{equation}
This, together with H$\mathrm{\ddot{o}}$lder inequality and Poincar$\mathrm{\acute{e}}$ inequality, gives
\begin{equation}\label{x3}
\begin{split}
\iint_{D_{\mathrm{r}}(\mathrm{x}_0)}\mathcal{J}_2\mathrm{d}r\mathrm{d}\theta&
\leq C_2
\bigg(\mathrm{r}^2\iint_{D_{\mathrm{r}}(\mathrm{x}_0)}
\vert\nabla\Psi^2\vert\mathrm{d}r\mathrm{d}\theta\bigg)^{\frac{1}{2}}
\bigg(\iint_{D_{\mathrm{r}}(\mathrm{x}_0)}
\vert\nabla Y_1\vert^2\mathrm{d}r\mathrm{d}\theta\bigg)^{\frac{1}{2}}\\
&\leq C_2\mathrm{r}
\Vert\Psi\Vert_{H^1(\Omega)}
\bigg(\iint_{D_{\mathrm{r}}(\mathrm{x}_0)}
\vert\nabla Y_1\vert^2\mathrm{d}r\mathrm{d}\theta\bigg)^{\frac{1}{2}},
\end{split}
\end{equation}
where $C_2=\Vert b_1\Vert_{C^0(\overline{\Omega})}$. Similarly, for $\mathcal{J}_3$, we have
\begin{equation}\label{x4}
\begin{split}
\iint_{D_{\mathrm{r}}(\mathrm{x}_0)}\mathcal{J}_3\mathrm{d}r\mathrm{d}\theta&
\leq C_3
\bigg(\iint_{D_{\mathrm{r}}(\mathrm{x}_0)}
\vert\nabla\varphi^2\vert\mathrm{d}r\mathrm{d}\theta\bigg)^{\frac{1}{2}}
\bigg(\mathrm{r}^2\iint_{D_{\mathrm{r}}(\mathrm{x}_0)}
\vert\nabla Y_2\vert^2\mathrm{d}r\mathrm{d}\theta\bigg)^{\frac{1}{2}}\\
&+C_3\bigg(\mathrm{r}^2\iint_{D_{\mathrm{r}}(\mathrm{x}_0)}
\vert\nabla\Psi^2\vert\mathrm{d}r\mathrm{d}\theta\bigg)^{\frac{1}{2}}
\bigg(\iint_{D_{\mathrm{r}}(\mathrm{x}_0)}
\vert\nabla Y_2\vert^2\mathrm{d}r\mathrm{d}\theta\bigg)^{\frac{1}{2}}\\
&\leq C_3\mathrm{r}
(\Vert\varphi\Vert_{H^1(\Omega)}+\Vert\Psi\Vert_{H^1(\Omega)})
\bigg(\iint_{D_{\mathrm{r}}(\mathrm{x}_0)}
\vert\nabla Y_2\vert^2\mathrm{d}r\mathrm{d}\theta\bigg)^{\frac{1}{2}},
\end{split}
\end{equation}
provided $C_3=\Vert c_1\Vert_{C^0(\overline{\Omega})}+\Vert c_2\Vert_{C^0(\overline{\Omega})}$.
On the other hand, one can obtain that there exists a constant $C_4>0$ such that
\begin{equation}\label{x5}
\langle(F_1,F_2,F_3),(z_1,z_2)\rangle
\leq C_4\Vert(F_1,F_2,F_3)\Vert_{0;\overline\Omega}\mathrm{r}\bigg(\iint_{D_{\mathrm{r}}(\mathrm{x}_0)}(\vert\nabla Y_1\vert^2+
\vert\nabla Y_2\vert^2)\mathrm{d}r\mathrm{d}\theta\bigg)^{\frac{1}{2}}.
\end{equation}
\par It follows from \eqref{x1}, \eqref{x2}, \eqref{x3}, \eqref{x4}, \eqref{x5} and Proposition \ref{prop1} that
\begin{equation}\label{x6}
\begin{split}
\iint_{D_{\mathrm{r}}(\mathrm{x}_0)}
\left(\vert\nabla Y_1\vert^2+\vert\nabla Y_2\vert^2\right)
\mathrm{d}r\mathrm{d}\theta
\leq C_5 \mathrm{r}^2 \Vert(F_1,F_2,F_3)\Vert_{0;\overline\Omega}^2,
\end{split}
\end{equation}
where $C_5>0$ is a constant.
We adjust the proof of Lemma 4.12 in \cite{Han} to obtain that there exist constants $\hat{\alpha}\in(0,1)$ and $\hat{C}>0$ such that $\omega_1$ and $\omega_{2}$ in \eqref{weak1} satisfy
\begin{equation}\label{x7}
\iint_{D_{\mathrm{s}}(\mathrm{x}_0)}(\vert\nabla\omega_1\vert^2+\vert\nabla\omega_2\vert^2)\mathrm{d}r\mathrm{d}\theta
\leq \hat{C}\left(\frac{\mathrm{s}}{\mathrm{r}}\right)^{2\hat{\alpha}}
\iint_{D_{\mathrm{r}}(\mathrm{x}_0)}(\vert\nabla\omega_1\vert^2+\vert\nabla\omega_2\vert^2)\mathrm{d}r\mathrm{d}\theta
\end{equation}
for any $\mathrm{s}\in(0,\mathrm{r}]$. In view of $(Y_1,Y_2)=(\varphi-\omega_1,\Psi-\omega_2)$, then
\begin{equation}
\begin{split}
\iint_{D_{\mathrm{s}}(\mathrm{x}_0)}(\vert\nabla\varphi\vert^2
+\vert\nabla\Psi\vert^2)\mathrm{d}r\mathrm{d}\theta
\leq& \hat{C}\left(\frac{\mathrm{s}}{\mathrm{r}}\right)^{2\hat{\alpha}}
\iint_{D_{\mathrm{r}}(\mathrm{x}_0)}(\vert\nabla\varphi\vert^2
+\vert\nabla\Psi\vert^2)\mathrm{d}r\mathrm{d}\theta\\
&+ 2C_5 \mathrm{r}^2 \Vert(F_1,F_2,F_3)\Vert_{0;\overline\Omega}^2,
\end{split}
\end{equation}
provided $C_5$ and $\hat{C}$ are same as in \eqref{x6} and \eqref{x7}, respectively. Since $\mathrm{r}\in(0,\frac{1}{10}\min\{1,R\}]$, we can choose a constant $\bar{\alpha}\in(0,\hat{\alpha})$ such that $\mathrm{r}^2<\mathrm{r}^{2\bar{\alpha}}<\mathrm{r}^{2\hat{\alpha}}$. According to Lemma 3.4 in \cite{Han}, for $R_0\in(0,\frac{1}{10}\min\{1,R\}]$, one has
\begin{equation}\label{xx}
\begin{split}
\iint_{D_{\mathrm{r}}(\mathrm{x}_0)}(\vert\nabla\varphi\vert^2
+\vert\nabla\Psi\vert^2)\mathrm{d}r\mathrm{d}\theta
\leq C\mathrm{r}^{2\bar{\alpha}}\Vert(F_1,F_2,F_3)\Vert_{0;\overline\Omega}^2
\end{split}
\end{equation}
for any $\mathrm{x}_0=(r_0,\theta_0)\in\Gamma_{p}\cap\Gamma_{ex}$ provided $r_0=R$ and $0<\mathrm{r}\leq R_0$. Repeating the argument above implies that \eqref{xx} holds for any given $\mathrm{x}_0\in\overline{\Omega}$.
\end{proof}

The $C^{1,\alpha}$-estimates in Proposition \ref{prop3} can be proved by adjusting arguments of Lemma 3.6 in \cite{Bae3}. Here we give a brief sketch for the proof.
\begin{proposition}\label{prop3}
Let $(\varphi,\Psi)\in\mathcal{R}$ be as in Proposition \ref{prop1}. For any $\alpha\in(0,1)$, there exists a constant $C_a$ depending on the background data and $\alpha$ such that $(\varphi,\Psi)$ satisfies
\begin{equation}\label{esma1}
\Vert(\varphi,\Psi)\Vert_{1,\alpha;\Omega}
\leq C_a\Vert(F_1,F_2,F_3)\Vert_{0,\alpha;\overline{\Omega}}.
\end{equation}
\end{proposition}
\begin{proof}
\textbf{Step 1: $C^{1,\alpha}$ estimate for $\varphi$.} By choosing $\omega=0$ in \eqref{w1}, $\varphi$ is regarded as a weak solution to
\begin{equation}\label{www}
\begin{split}
\iint_{\Omega}\bigg(a_{11}\varphi_r\xi_r+a_{22}\varphi_{\theta}\xi_{\theta}
\bigg)\mathrm{d}r\mathrm{d}\theta
=\iint_{\Omega}((F_1-b_1\Psi)\xi_r+F_2\xi_{\theta})\mathrm{d}r\mathrm{d}\theta
-\int_{\Gamma_{ex}}F_1\xi\mathrm{d}\theta
-\int_{\Gamma_{\omega}}F_2\xi\mathrm{d}r.
\end{split}
\end{equation}
Using \eqref{H^1}, \eqref{C^c} and the method of reflection with respect to $\Gamma_{\omega}$ gives
\begin{equation}\label{C^11}
\Vert\varphi\Vert_{1,\alpha;\overline\Omega}
\leq C \Vert(F_1,F_2,F_3)\Vert_{0,\alpha;\overline{\Omega}},
\end{equation}
provided $C>0$ is a constant depending on the background data and $\alpha$.

\textbf{Step 2: $C^{1,\alpha}$ estimate for $\Psi$.}
Substituting $\xi=0$ into \eqref{w1}, then $\Psi$ is regarded as a weak solution to
\begin{equation}
\begin{split}
\iint_{\Omega}\bigg(\hat{r}\Psi_r\omega_r
+\frac{1}{\hat{r}}\Psi_{\theta}\omega_{\theta}
-c_2\Psi\omega\bigg)\mathrm{d}r\mathrm{d}\theta
=\iint_{\Omega}(c_1\varphi_r-F_3)\omega\mathrm{d}r\mathrm{d}\theta.
\end{split}
\end{equation}
Using Lemma \ref{lem2}, \eqref{C^c} and \eqref{C^11} yields
\begin{equation*}
\Vert (c_1\varphi_r-F_3)\Vert_{0,\alpha;\overline{\Omega}}
\leq C\Vert(F_1,F_2,F_3)\Vert_{0,\alpha;\overline{\Omega}}.
\end{equation*}
With the help of the compatibility conditions \eqref{compati} and the method of reflection with respect to $\Gamma_{\omega}$, one obtains
\begin{equation}\label{C^22}
\Vert\Psi\Vert_{1,\alpha;\overline\Omega}
\leq C \Vert(F_1,F_2,F_3)\Vert_{0,\alpha;\overline{\Omega}},
\end{equation}
where the constant $C>0$ depends on the background data and $\alpha$. This finishes the proof of Proposition \ref{prop3}.
\end{proof}

\begin{proposition}\label{prop4}
Fix $\alpha\in(0,1)$. The linear boundary value problem \eqref{elli1} and \eqref{aBC1} has a unique solution $(\mathcal{U}, \mathcal{V}, \check{\Phi})\in[C^{2,\alpha}(\overline{\Omega})]^3$ satisfying the estimate
\begin{equation}\label{C2c}
\begin{split}
\Vert(\mathcal{U}, \mathcal{V}, \check{\Phi})\Vert_{2,\alpha;\overline\Omega}
\leq & C_b\Big(\Vert(f_1,f_2)\Vert_{2,\alpha;\overline{\Omega}}
+\Vert f_3\Vert_{0,\alpha;\overline{\Omega}}
+\Vert f_4\Vert_{1,\alpha;\overline{\Omega}}
+\sigma_p\Big)\leq C_b(\delta^2+\sigma_p),
\end{split}
\end{equation}
where  $C_b>0$ is a constant depending on the background data and $\alpha$. Furthermore, the solution $(\mathcal{U}, \mathcal{V}, \check{\Phi})$ satisfies the compatibility conditions
\begin{equation}\label{compaC2}
\partial_{\theta}\mathcal{U}(r,\pm \theta_0)=\mathcal{V}(r,\pm \theta_0)=\partial^2_{\theta}\mathcal{V}(r,\pm \theta_0)=\partial_{\theta}\check{\Phi}(r,\pm \theta_0)=0, \ \ \hbox{for} \ \ r\in[0,R].
\end{equation}
\end{proposition}
\begin{proof}
\textbf{Step 1: $C^{2,\alpha}$ estimate for $\check{\Phi}$.}
It follows from \eqref{transUV1}, \eqref{transf1}, \eqref{transUV2}, \eqref{newva1}, \eqref{aBC1-e-w} and \eqref{esma1} that
\begin{equation}
\Vert(\mathcal{U}, \mathcal{V}, \check{\Phi})\Vert_{1,\alpha;\overline\Omega}
\leq C_e\Big(\Vert(f_1,f_2)\Vert_{2,\alpha;\overline\Omega}
+\Vert f_3\Vert_{0,\alpha;{\overline\Omega}}
+\Vert f_4\Vert_{1,\alpha;{\overline\Omega}}+\sigma_p\Big)\leq C_e(\delta^2+\sigma_p),
\end{equation}
provided  $C_e$ is a constant. With this estimate, we consider the elliptic equation
\begin{equation}
\partial_r(\hat{r}\check{\Phi}_{r})+\frac{1}{\hat{r}}\partial_{\theta}^2\check{\Phi}
+c_2(r)\check{\Phi}
=f_3-c_1(r)\mathcal{U},
\end{equation}
with the boundary conditions
\begin{equation}
\begin{cases}
\check{\Phi}=\Phi_{en}-\bar \Phi(0),\ \ &{\rm{on}}\ \ \Gamma_{en},\\
\check{\Phi}=\Phi_{ex}-\bar \Phi(R),\ \ &{\rm{on}}\ \ \Gamma_{ex},\\
\partial_{\theta}\check{\Phi}=0, \ \ &{\rm{on}}\ \ \Gamma_\omega.
\end{cases}
\end{equation}
Employing the compatibility condition \eqref{compatibilityf}, the standard Schauder estimate in \cite{Gilbarg} and the method of reflection deduces that
\begin{equation}\label{PhiC2}
\Vert\check{\Phi}\Vert_{2,\alpha;\overline\Omega}
\leq C_{\mathcal{A}} (\Vert f_3\Vert_{1,\alpha;\overline{\Omega}}+\Vert \mathcal{U}\Vert_{1,\alpha;\overline{\Omega}}+\sigma_p)\leq C_{\mathcal{A}}(\delta^2+\sigma_p),
\end{equation}
where $C_{\mathcal{A}}$ is a constant depending only on background data and $\alpha$.

\textbf{Step 2: $C^{2,\alpha}$ estimate for $(\mathcal{U},\mathcal{V})$.}
For the linear system
\begin{equation}
\begin{cases}
\partial_{r}(A_{11}(r)\mathcal{U})+\partial_{\theta}(
A_{22}(r)\mathcal{V})=\partial_r (f_1-b_1(r)\check{\Phi})+\partial_{\theta} f_2,\\
\partial_{r}(\hat{r}\mathcal{V})-\partial_{\theta}\mathcal{U}=f_4,
\end{cases}
\end{equation}
with the boundary conditions
\begin{equation}
\begin{cases}
\mathcal{V}=V_{en},\ \ &{\rm{on}}\ \ \Gamma_{en},\\
\mathcal{U}=h,\ \ &{\rm{on}}\ \ \Gamma_{ex},\\
\mathcal{V}=0, \ \ &{\rm{on}}\ \ \Gamma_\omega.
\end{cases}
\end{equation}
The standard symmetric extension technique, together with the standard Schauder estimate in \cite{Gilbarg} again, yields that $(\mathcal{U},\mathcal{V})\in \left(C^{2,\alpha}(\overline{\Omega})\right)^2$. Furthermore, it holds that
\begin{equation}\label{UVC2}
\Vert(\mathcal{U},\mathcal{V})\Vert_{2,\alpha;\overline\Omega}
\leq C_{\mathcal{B}} (\Vert (f_1,f_2,\check{\Phi})\Vert_{2,\alpha;\overline{\Omega}}+\Vert f_4\Vert_{1,\alpha;\overline{\Omega}}+\sigma_p)\leq C_{\mathcal{B}}(\delta_2+\sigma_p),
\end{equation}
provided the constant $C_{\mathcal{B}}$ depends only on background data and $\alpha$. Moreover, \eqref{compaC2} can be also verified. Consequently, the linear boundary value problem \eqref{elli1} and \eqref{aBC1} has a unique solution $(\mathcal{U}, \mathcal{V}, \check{\Phi})\in[C^{2,\alpha}(\overline{\Omega})]^3$ that is actually a classical solution. This finishes the proof of Proposition \ref{prop4}.
\end{proof}

\section{Proof of Theorem \ref{thm1}}\noindent
\par For  any given  $\mathbf{V}^{\sharp}\in \mathcal{J}(\delta)  $, by \eqref{3-7} and Proposition \ref{prop4}, we have obtained a new  $\mathbf{V}$, which satisfies
\begin{equation}\label{es1}
\Vert\mathbf{V}\Vert_{2,\alpha;\overline\Omega}
\leq \mathcal{C}_1(\delta^2+\sigma_p),
\end{equation}
where $\mathcal{C}_1>0$ is a constant depending on the background data and $\alpha$. Define a map $ \mathcal{T} $ as follows
\begin{equation}\label{4-101}
\mathcal{T}(\mathbf{V}^{\sharp})=\mathbf{V},
\ \ {\rm{ for }} \ \mathbf{V}^{\sharp}\in \mathcal{J}(\delta).
\end{equation}
  Let $\sigma_1=\frac{1}{4(\mathcal{C}_1^2+1)}$ and choose $\delta= 2\mathcal{C}_1\sigma_p$. Then if $ \sigma_p\leq \sigma_1 $, one has
\begin{equation*}
\begin{aligned}
\Vert\mathbf{V}\Vert_{2,\alpha;\overline\Omega}\leq \delta. \end{aligned}
\end{equation*}
This implies $\mathcal{T}$ maps $\mathcal{J}(\delta)$ into itself. Next, we will show that $\mathcal{T}$ is a contraction mapping in $\mathcal{J}(\delta) $. Set $\mathbf{V}_{i}^{\sharp} \in \mathcal{J}(\delta)$,  $i=1,2$, then $\mathbf{V}_{i}=\mathcal{T}(\mathbf{V}_{i}^{\sharp}) $. Define
\begin{equation*}
\mathbf{Y}^{\sharp}=\mathbf{V}_{1}^{\sharp}-\mathbf{V}_{2}^{\sharp},\quad
\mathbf{Y}=\mathbf{V}_{1}-\mathbf{V}_{2}.
\end{equation*}
It follows from \eqref{St} and \eqref{Kt} that $Y_4=\mathcal{S}_1-\mathcal{S}_2$ and $Y_5=\mathcal{K}_1-\mathcal{K}_2$ satisfy
\begin{equation}\label{St-d}
\begin{cases}
\bigg(\partial_{r}+\frac{\mathcal{V}_1^{\sharp}}{\hat{r}(\bar U+\mathcal{U}_1^{\sharp})}\partial_{\theta} \bigg) Y_4=\bigg(\frac{\mathcal{V}_2^{\sharp}}{\hat{r}(\bar U+\mathcal{U}_2^{\sharp})}-\frac{\mathcal{V}_1^{\sharp}}{\hat{r}(\bar U+\mathcal{U}_1^{\sharp})}\bigg)\partial_{\theta}\mathcal{S}_2,\\
Y_4(0,\theta)=0,
\end{cases}
\end{equation}
and
\begin{equation}\label{Kt-d}
\begin{cases}
\bigg(\partial_{r}+\frac{\mathcal{V}_1^{\sharp}}{\hat{r}(\bar U+\mathcal{U}_1^{\sharp})}\partial_{\theta} \bigg) Y_5=\bigg(\frac{\mathcal{V}_2^{\sharp}}{\hat{r}(\bar U+\mathcal{U}_2^{\sharp})}-\frac{\mathcal{V}_1^{\sharp}}{\hat{r}(\bar U+\mathcal{U}_1^{\sharp})}\bigg)\partial_{\theta}\mathcal{K}_2,\\
Y_5(0,\theta)=0.
\end{cases}
\end{equation}
By the characteristic method, one has
\begin{equation}\label{4-105}
\Vert(Y_4,Y_5)\Vert_{1,\alpha;\overline\Omega}\leq C\Vert\mathbf{V}_2\Vert_{2,\alpha;\overline\Omega} \Vert\mathbf{Y}^{\sharp}\Vert_{1,\alpha;\overline\Omega}
\leq C\delta \Vert\mathbf{Y}^{\sharp}\Vert_{C^{1,\alpha}(\overline\Omega)}.
\end{equation}
Similarly, it follows from \eqref{elli1-d} and \eqref{NaBC1} that $Y_1=\mathcal{U}_1-\mathcal{U}_2$, $Y_2=\mathcal{V}_1-\mathcal{V}_2$ and $Y_3=\check{\Phi}_1-\check{\Phi}_2$ satisfy the equations
\begin{equation}\label{elli1-d}
\begin{cases}
\partial_{r}(A_{11}(r)Y_1)+\partial_{\theta}(
A_{22}(r)Y_2)+\partial_{r}(b_1(r)Y_3)\\
=\partial_r (f_1^{1}-f_1^{2})+\partial_{\theta}(f_2^{1}-f_2^{2}),\\
\partial_r(\hat{r}(Y_3)_r)+\frac{1}{\hat{r}}\partial_{\theta}^2Y_3
+c_1(r)Y_1+c_2(r)Y_3
=f_3^{1}-f_3^{2},\\
\partial_{r}(\hat{r}Y_2)-\partial_{\theta}Y_1=f_4^{1}-f_4^{2},
\end{cases}
\end{equation}
with the boundary conditions
\begin{equation}\label{NaBC1}
\begin{cases}
(Y_2,Y_3)=(0,0),\ \ &{\rm{on}}\ \ \Gamma_{en},\\
(Y_1,Y_3)=(h^{1}-h^{2},0),\ \ &{\rm{on}}\ \ \Gamma_{ex},\\
Y_2=\partial_{\theta}Y_3=0, \ \ &{\rm{on}}\ \ \Gamma_\omega,
\end{cases}
\end{equation}
where
\begin{equation*}
f_i^{ k}=f_i(r,\theta;\mathcal{U}^{\sharp k},\mathcal{V}^{\sharp k},\check{\Phi}^{\sharp k},\mathcal{S}^{ k},\mathcal{K}^{ k}) \ \
{\rm {and}} \ \
h^{ k}=h(\theta;\mathcal{U}^{\sharp k},\mathcal{V}^{\sharp k},\mathcal{S}^{ k},\mathcal{K}^{ k})
\end{equation*}
for $i=1,2,3,4$ and $k=1,2$. Then we have the following estimate
\begin{equation}\label{4-107}
\begin{split}
\Vert(Y_1,Y_2,Y_3)\Vert_{1,\alpha;\overline\Omega}
\leq & C\bigg( \Vert(f_1^{1},f_2^{1})-(f_1^{2},f_2^{2})\Vert_{1,\alpha;\overline\Omega}
+\Vert (f_3^{1},f_4^{1})-(f_3^{2},f_4^{2})\Vert_{0,\alpha;\overline\Omega}\\
&\quad+\Vert h^{1}-h^{2}\Vert_{1,\alpha;[-\theta_0,\theta_0]}\bigg)
\leq C \delta \Vert\mathbf{Y}^{\sharp}\Vert_{1,\alpha;\overline\Omega}.
\end{split}
\end{equation}
Collecting all the above estimates leads to
\begin{equation}\label{contraction}
\Vert\mathbf{Y}\Vert_{1,\alpha;\overline\Omega}\leq \mathcal{C}_{2}\delta\Vert\mathbf{Y}^{\sharp}\Vert_{1,\alpha;\overline\Omega},
\end{equation}
provided the constant $\mathcal{C}_{2}>0$.
Set $\sigma_2=\min\left\{\sigma_1,\frac{1}{3\mathcal{C}_1\mathcal{C}_2}\right\}$. Then for any $\sigma\leq \sigma_2$, it holds that $\mathcal{C}_{2}\delta= 2\mathcal{C}_1 \mathcal{C}_2\sigma\leq \frac{2}{3}$. Therefore, $\mathcal{T}$ is a contraction mapping with respect to $C^{1,\alpha}$ norm so that there exists a unique fixed point $\mathbf{V}=(\mathcal{U},\mathcal{V},\check{\Phi},\mathcal{S},\mathcal{K})\in \mathcal{J}(\delta)$. This completes the proof of Theorem \ref{thm1}.
\par{\bf Acknowledgement.}
 The research of Yuanyuan Xing is partially supported by  the Natural Science Foundation of Hebei province, China (No. A2025501003) and Fundamental Research Funds for the Central Universities (N2523027). The research of Zihao Zhang was supported by  the Postdoctoral Fellowship Program of CPSF under Grant Number GZB20250719.
\par {\bf Data availability.} No data was used for the research described in the article.
    \par {\bf Conflict of interest.} This work does not have any conflicts of interest.

\end{document}